\documentclass[11pt,a4paper]{article}
\usepackage[unicode]{hyperref}

\usepackage{tikz}\usetikzlibrary{arrows}
\usepackage{pgfplots}\usetikzlibrary{plotmarks}
\pgfplotsset{compat=1.16} 

\usepackage{url}
\usepackage{graphicx}
\usepackage{tabularx}
\usepackage{tikz}
\usepackage{soul} 
\usepackage{sidecap}
\usepackage{subfig}
\usepackage{epstopdf}
\usepackage{wrapfig}
\usepackage{natbib}
\usepackage{amssymb,amsmath}
\usepackage{bbm}
\usepackage[cp1251]{inputenc}  
 \usepackage[norelsize,ruled, lined]{algorithm2e}

\newtheorem{theorem}{Theorem}

\newtheorem{problem}[theorem]{Problem}

\newtheorem{proposition}{Proposition}

\DeclareMathOperator{\diag}{diag}
 \DeclareMathOperator*{\argmin}{arg\,min}

\newenvironment{proof}[1][Proof]{\textbf{#1.} }{\ \rule{0.5em}{0.5em}}
\usepackage[font=small]{caption}

\title{Mixed-Integer Approaches to Constrained \\Optimum Communication Spanning Tree Problem}
\author{Alexander Veremyev$^a$, Mikhail Goubko$^b$}
\date{\small
$^a$~University of Central Florida, Orlando, FL, 32816-2368, USA, alexander.veremyev@ucf.edu\\ 
$^b$~V.\,A.~Trapeznikov Institute of Control Sciences, 117997, Profsoyuznaya, 65, Moscow, Russia, mgoubko@mail.ru (\textbf{corresponding author})
}

\providecommand{\keywords}[1]
{
  \small	
  \textbf{\textit{Keywords---}} #1
}

\begin{document}
\maketitle

\begin{abstract}
Several novel mixed-integer linear and bilinear formulations are proposed for the optimum communication spanning tree problem. They implement the distance-based approach: graph distances are directly modeled by continuous, integral, or binary variables, and interconnection between distance variables is established using the recursive Bellman-type conditions or using matrix equations from algebraic graph theory. These non-linear relations are used either directly giving rise to the bilinear formulations, or, through the big-$M$ reformulation, resulting in the linear programs. A branch-and-bound framework of Gurobi~$9.0$ optimization software is employed to compare performance of the novel formulations on the example of an optimum requirement spanning tree problem with additional vertex degree constraints. Several real-world requirements matrices from transportation industry are used to generate a number of examples of different size, and computational experiments show the superiority of the two novel linear distance-based formulations over the the traditional multicommodity flow model.
\end{abstract}


\section{Introduction}\label{sect_intro}




To simplify notation, in this article an undirected graph $G=\langle V,E\rangle$ with a vertex set $V$ and an edge set $E$ is modeled as a digraph with bi-directional arcs; so, if an edge exists in $G$ between vertices $i$ and $j$, then both $(i,j)\in E$ and $(j,i)\in E$. 

Given an undirected graph $G=\langle V,E_s\rangle$ with the non-negative length $t_{ij}=t_{ji}$ assigned to each admissible edge $(i,j)\in E_s$ and a set of non-negative requirements $\mu_{ij}=\mu_{ji}$, $(i,j)\in R\subseteq V\times V$, $i\neq j$, the \textit{optimum communication spanning tree problem} (OCSTP) is to find a spanning tree $T$ of the graph $G$ that minimizes the \textit{communication cost}
\begin{equation}\label{eq_cost}
C(T)=\sum_{i,j\in V, i\neq j}\mu_{ij}d_T(i,j),
\end{equation}
where $d_T(i,j)$ is the total length of the path between the vertices $i$ and $j$ in the tree $T$.

Additional constraints imposed on the topology of admissible trees (e.g., on the maximum vertex degree \citep{kim2015distributed} or the whole degree sequence \citep{goubko2016minimizing}) give rise to the \textit{constrained OCSTP}.

OCSTP is among classical NP-complete problems of combinatorial optimization \citep{Garey79} arising in many areas, including communications, transportation industry, mathematical chemistry, and molecular biology. 

We use Gurobi\textsuperscript{TM} MILP/MIBP solver to compare traditional mixed-integer programming (MIP) approaches to the constrained OCSTP (the brute-force and the multicommodity flow schemes) with some recent mixed-integer linear programming (MILP) and quadratic (MIQP) formulation techniques \citep{veremyev2015critical}. We also propose one novel MILP scheme and two novel MIQP schemes that are based on the direct modeling of distances in a tree. 

The computational evidence is presented on a broad set of origin-destination matrices from the transportation industry. We show that the novel MILP formulations significantly outperform their competitors allowing to build optimal trees with the given vertex degree sequence over several dozens of vertices on a PC using the universal algorithms implemented in modern commercial solvers. We also discuss the limits of possible extension of the proposed models to directed graphs and to general graphs with cycles.

\section{Literature review}\label{sect_review}

Proposed by \cite{hu1974optimum}, OCSTP reduces to the more general \textit{multicommodity uncapacitated fixed-charge network design problem} \citep{wong1976survey}, which belongs to a broad class of \textit{optimal network design problems} (ONDP) \citep{scott1969optimal}. OCSTP has important special cases. It is called the \textit{optimum distance spanning tree problem} (ODSTP) if requirements $\mu_{ij}$ are all assumed to be equal; and it is called the \textit{optimum requirement spanning tree problem} (ORSTP) if equal edge lengths $t_{ij}$ are defined.

Both ODSTP and ORSTP (and, hence, OCSTP) are known to be NP-complete \citep{johnson1978complexity}. In the proof, the graph $G$ is built with topology that resembles the reduction from some known NP-complete problem. However, additional conditions on vertex degrees  may complicate even ORSTP over a complete graph $G$. An exact algorithm \citep{hu1974optimum} of complexity $\Theta(|V|^4)$ exists for this problem\footnote{	$f(n)=\Theta(g(n))$ if and only if $f$ is bounded both above and below by $g$ asymptotically.}, but for spanning trees with the fixed vertex degree sequence this problem is NP-complete (see Theorem \ref{th_NP} in the next section). Only in the special case of \textit{product-requirements} (the requirements that can be represented in a form $\mu_{ij}=\mu_i\cdot\mu_j$), the optimal spanning tree with the fixed vertex degree sequence 
can be effectively identified using the generalized Huffman algorithm \citep{goubko2016minimizing}. 

Due to its complexity, the study of OCSTP was concentrated on heuristics and on the implicit enumeration algorithms. The first exhaustive enumeration approach that used the backtrack programming was devised in \cite{scott1969optimal}. With the two early branch-and-bound schemes \citep{ahuja1987exact,dionne1979exact}, several individual instances of moderate size (up to $|V|=40, |E_s|=69$) were solved to optimality using the processing capacity available at that time. 

Despite the longstanding interest to OCSTP, just a few MIP formulations were proposed in the literature. Two different MILP formulations are due to 
\cite{fischetti2002exact}. The \textit{path-based formulation}\footnote{The notation is incoherent between authors. We follow the notation from \cite{masone2019minimum} for consistency.} involves the exponential number of variables and $\Theta(|E_s||R|)$ constraints. The \textit{flow-based formulation} is based on the most general (and probably the most popular) characterization of multicommodity network flows modelled using continuous flow variables $u_{ij}^{st}$ (the amount of the requirement $(s,t)\in R$ that goes through the edge $(i,j)\in E_s$). It involves $\Theta(|E_s||R|)$  variables and $\Theta(|E_s||R|)$ constraints. Recently 
\cite{zetina2019solving} managed to solve instances with up to 60 vertices using the tailored branch-and-cut scheme based on this model. In the \textit{reduced flow-based model} \citep{FERNANDEZ201385} the number of variables is scaled down to $\Theta(|E_s||V|)$ by aggregating flow variables sharing the same origin. This formulation is shown to have the weaker relaxation than the original flow-based model \citep{luna2016optimum}. 

Both the path-based and the flow-based model easily generalize to communication networks with cycles \citep{zetina2019exact}, unlike the \textit{rooted tree formulation} proposed in \cite{luna2016optimum}. This compact MILP model involves just $\Theta(|V|^2)$ variables with $\Theta(|V|^3)$ constraints, and its relaxation is compatible with that of the reduced flow-based formulation. 

The flow-based model is also used to solve the hop-constrained spanning tree problem, which is a variant of the constrained ORSTP where the network construction cost \citep{gouveia1996multicommodity} or a combination of construction and communication costs \citep{Pirkul03} are minimized over the set of spanning trees with the limited diameter. 

The constrained OCSTP is mainly examined in the context of the \textit{Wiener index} studies \citep{dobrynin2001wiener} in the mathematical chemistry.\footnote{The Wiener index coincides with the communication cost (\ref{eq_cost}) under $\mu_{ij}=1$ and $t_{ij}=1$.} Over the years, many lower bounds (including some tight bounds) are obtained for the Wiener index over different subdomains of trees, e.g., trees with the maximum degree \citep{FISCHERMANN2002127}, the given degree sequence \citep{zhang2008wiener}, the eccentricity sequence \citep{Dankelmann2020611}, or the segment sequence \citep{zhang2019wiener}) and general graphs (chemical graphs \citep{Knor2019119}, unicyclic graphs with the given girth \citep{yu2010wiener}, matching number \citep{du2010minimum}, etc.). Some analytical results of this theory generalize to ODSTP \citep{gao2015vertex,wang2012wiener} and ORSTP \citep{goubko2020lower}.

In the present article we combine computational tools of the combinatorial optimization and integer programming with the methods of algebraic graph theory to propose the novel high-performance MIP formulations for OCSTP with vertex degree constrains and to compare them with the traditional approaches.

\section{Methodology}

MIP formulations for OCSTP studied in the next section may have different limitations. Some of them are applied only to ORSTP, while the others hardly generalize to graphs with cycles. For the direct comparison of all formulations, a simplistic ORSTP over a complete graph is used as a common base.  Given a symmetric non-negative requirements matrix $A=(\mu_{ij})_{i,j=i\neq j}^n$, the problem is to find a tree $T$ over the vertex set $V=\{1,...,n\}$ that minimizes $\sum_{i,j\in V,i\neq j}\mu_{ij}d_T(i,j)$, where $d_{ij}(T)$ is the number of edges in the path between vertices $i$ and $j$ in $T$. To test the robustness of different MIP formulations with respect to additional constraints, we allow only for trees with the given vertex degree sequence $d=(d_1,...,d_n)$. So, the problem instance is defined by the requirements matrix $A$ and the vertex degree sequence $d$. 

Vertex degree constraints are quite natural in the context of ORSTP. For example, in the communication network topology design distances measure the ``hop count'' (the number of intermediate devices) between vertices in a network, routers have the limited number of ports, while client terminals support just a single connection to a router.  

Without vertex degree constraints ORSTP over the complete graph is solvable in polynomial time \citep{hu1974optimum}. However, addition of degree constraints makes it NP-complete.

\begin{theorem}\label{th_NP}
ORSTP with the fixed vertex degree sequence is NP-complete. 
\end{theorem}

\begin{proof}
Let us consider a special case of ORSTP for trees with $2m$ leaves and just two internal vertices of equal degree. These trees have the degree sequence $d=(1, ..., 1, m+1, m+1)$. Let us assume that non-zero requirements exist only between leaves, according to the matrix $A=(\mu_{ij})_{i,j=1}^{2m}$. The distance $d_T(i,j)$ between the leaves $i$ and $j$ in such a tree $T$ is equal to 2 if they are connected to the same internal vertex, and is equal to 3 otherwise. 
ORSTP in this case reduces to the selection of a set of $m$ leaves $S\subseteq \{1,...,2m\}$ to be connected to the first internal vertex, while the rest of leaves are connected to the second internal vertex. The communication cost can be written as
$$\sum_{i,j\in V,i\neq j}\mu_{ij}d_T(i,j)=2\sum_{i,j=1,i\neq j}^{2m}\mu_{ij}+\text{cut}_A(S),$$
where $\text{cut}_A(S):=\sum_{i\in S, j\notin S} \mu_{ij}$ is the \textit{graph cut} of a weighted graph with $2m$ vertices and the adjacency matrix $A$.
Hence, the \textit{balanced graph bisection problem} $\min_{S:|S|=m}\text{cut}_A(S)$, which is known to be NP-complete \citep{Garey79}, reduces to the ORSTP with a special vertex degree sequence.

Conversely, the set of trees with the given vertex degree sequence is a subset of a wider set of all spanning trees. So, ORSTP with the fixed vertex degree sequence is NP, and, hence, is NP-complete.
\end{proof}

The present article aims at providing different MIP formulations and illustrating their performance, so we also need a common algorithmic base. We use a universal branch-and-bound scheme of a commercial MILP/MIQP solver (Gurobi 9.0, which is among the fastest MILP solvers of the moment that also support MIQP). Although substantial improvement can be obtained by designing custom algorithms for traditional formulations\footnote{For example, for the flow-based formulation, the hundredfold acceleration has been achieved in \cite{zetina2019solving} as compared to the state-of-the-art commercial solver (CPLEX 12)}, we believe that in the future similar efforts can accelerate the novel settings proposed in this article. For the same reason default Gurobi parameter values are used without any parameter tuning.

As an example of traditional MILP models, we consider the flow-based formulation, which is reported to be the most promising for standard solvers \citep{zetina2019solving}. In this article we do not employ the row/column generation, so we concentrate on the models with polynomial number of variables and constrains. For this reason we do not consider characterizations of spanning trees with Cut-set inequalities or with Subtour elimination constrains. We also exclude the path-based model from \cite{fischetti2002exact} due to the exponential number of variables. (It is also known to share the same continuous relaxation with the flow-based model \citep{fischetti2002exact}.) The reduced flow-based formulation is also excluded due to weaker continuous relaxations \citep{luna2016optimum}. 

In the flow-based formulation, the distance between vertices $s$ and $t$ is calculated implicitly as $\sum_{(i,j)\in E_s}t_{ij}u_{ij}^{st}$ from flow variables $u_{ij}^{st}$ (see Section~\ref{sect_flow} for details). Several recently appeared approaches that can explicitly model the distance matrix of a tree. As an example, we include the formulation based on the ideas from \cite{veremyev2015critical}, where distances are modeled with binary variables. Developed for the critical vertices detection, this model has, to our knowledge, never been applied to OCSTP, although it has been applied to some problems of netwrok design \citep{mukherjee2017minimum,diaz2019robust}.

In the compact rooted-tree formulation proposed in \cite{luna2016optimum} distances are modeled by continuous variables. We do not consider it because of its poorer relaxation \citep{luna2016optimum}. Instead, a novel, more straightforward MILP model is proposed, which also uses continuous distance variables and, despite the rooted-tree formulation, easily generalizes to graphs with cycles.

A yet another novel compact model with $\Theta(n^2)$ variables based on the recent characterization of Laplacian and distance matrices of trees by $\Theta(n^2)$ bilinear equations \citep{goubko2020bilinear} is developed and studied in Section~\ref{sect_LD}. 

We also include a brute-force approach that combines exhaustive enumeration of trees for internal vertices and  assignment of leaves from the solution of the quadratic assignment problem (QAP) under the fixed topology of the tree. It is used as a baseline when the other formulations are compared.

The latter two formulations reduce to MIQP. Typically, MIQP models are more compact. They involve less variables and constrains and become an attractive alternative to MILP in the view of recent advances in exact algorithms for convex and non-convex MIQP.

The efficiency of the branch-and-bound algorithm considerably increases when a promising candidate solution is provided. So we compare general-purpose heuristic algorithms for candidate solution generation provided by Gurobi with the custom local search procedure proposed in \cite{goubko2020lower}.

Communication spanning trees resemble the intrinsic structure of requirements by placing vertices $i$ and $j$ closer in a tree when the requirement $\mu_{ij}$ is high and divorcing the vertices with lower communication requirements. If elements of the requirements matrix $A$ are taken from realizations of independent random variables, requirements inherit no internal structure. Therefore, the use of completely random matrices in computational experiments with OCSTP algorithms might be irrelevant. To obtain more reasonable results, we generate connected requirements matrices of various size from the collection of eight distinct origin-destination matrices from public and private transportation (water, ground, and air). See \cite{goubko2020lower} for the detailed description of data sets.



\section{Compendium of MIP models}\label{sect_mip}

In this section we introduce MIP formulations and algorithms, while their comparison on the basis of computational experiments is postponed to  Section~\ref{sect_comp}. 

Since modern solvers (e.g., Gurobi\textsuperscript{TM} or Fico\textsuperscript{TM} Xpress) can handle both quadratic and linear MIP formulations, in  Section~\ref{sect_miqp} we start with MIQP and MIQCP (mixed-integer quadratically constrained program) models for the constrained OCSTP.

First in  Section~\ref{sect_brute} we introduce a rather straightforward enumerative approach, which combines exhaustive enumeration of tree topologies for internal vertices and assignment of leaves from the solution of QAP under the fixed topology of the tree. The idea is to show how OCSTP can be solved without developing any problem-specific approach but by combining two highly efficient general-purpose algorithms (enumeration of labeled trees using Pr{\"u}fer codes and QAP), which requires minimum additional coding. Any tailored algorithm should at least outperform this ideologically and technically simple approach.

    Then in  Sections~\ref{sect_LD} and \ref{sect_quad_binary} we present two different MIQCP approaches (based on integral and binary distance variables, respectively) which allow to model distances in a tree taking into account algebraic properties of this class of graphs. These novel MIQCP formulations have never been applied to OCSTP, and their comparative performance for this problem is the main concern.  

We proceed to MILP formulations in  Section~\ref{sect_milp}.  Section~\ref{sect_flow} introduces the classical multicommodity flow-based model, which lies in the core of most state-of-the-art MIP approaches to OCSTP. Its computation time is the main baseline for the performance measurement of all novel MIQP and MILP formulations. In  Sections~\ref{sect_f1l} and \ref{sect_f2l} we introduce two novel MILP formulations. The first formulation uses integral variables to model distances in a tree while the second one employs binary variables. 

The latter two formulations demonstrate superior performance in our computational experiments. Also, although they differ in the number of variables and constraints, the experiments indicate that none of them clearly dominates the other by its speed. 

In all formulations considered below, binary variables $x_{ij}$ model allowable edges $(i,j)\in E_s$. Therefore, $x_{ij}=1$ if an allowable edge $(i,j)\in E_s$ is included in the spanning tree, and $x_{ij}=0$ otherwise. In this article we focus only on undirected graphs, so $x_{ij}=x_{ji}$ for all $(i,j)\in E_s$, and variables $x_{ij}$ for $i>j$ are redundant. However, to simplify notation we use them keeping in mind that a solver will reduce them to $x_{ji}$ at the preprocessing stage. In Conclusion we discuss how our results extend to directed graphs. 

Clearly, we consider only \textit{arborescent} vertex degree sequences\footnote{Those having $d_1+...+d_n=2(n-1)$.}, which are compatible with the structure of graph $G$, so that $d_i\le |N_G(i)|$ for any vertex $i \in V$, where $N_G(i)$ is the neighborhood of the vertex $i$ in a graph $G$. Also, until a connected graph with at least three vertices is sought, we can discard edges between leaves and assume that if $d_i = d_j= 1$ then $(i,j)\notin E_s$. Moreover, any vertex of a connected graph is connected to at least one internal vertex, so we use inequalities $\sum_{j:d_j>1} x_{ij}\ge 1$, $i\in V$ to limit the search space. For the same reason, at most one leaf vertex can be connected to a vertex of degree 2, so if $d_i=2$ then $\sum_{j:d_j=1}x_{ij}\le 1$. These inequalities allow to further limit the search space.

This article focuses on the search of spanning trees. The sole path connects any pair of vertices in a tree, which can also be used to limit the search space. For example, for any edge $(i,j)$ in a tree the shortest path either from the vertex $i$ or from the vertex $j$ to an arbitrary vertex $k\in V$ traverses the edge $(i,j)$. Some formulations (e.g., see  Section~\ref{sect_LD}) explicitly use the fact that we seek for a tree. In Conclusion we discuss the price of extending the considered formulations to graphs with loops. 

Since we compare formulations on the example of ORSTP, the graph distance between vertices is just the number of edges in a path, so all distances are integral. If the degree sequence $d$ has $m$ internal vertices (i.e., components with $d_i>1$), distances belong to the range $\{1,..,L\}$, where $L := m+2$ is the maximum possible diameter of a spanning tree with the given degree sequence.  In Conclusion we discuss how to extend the formulations to the general OCSTP with integral and continuous edge weights $t_{ij}$, $(i,j)\in E_s$.  

All novel formulations presented in this article follow the common \textit{distance-based approach}. The distance between vertices $i,j\in V$ is modeled explicitly either using an integral variable $d_{ij}$ in MIQCP formulation of  Section~\ref{sect_LD} and in MILP formulation of  Section~\ref{sect_f1l} or using a binary variable $w_{ij}^{(\ell)}$ in MIQCP formulation of  Section~\ref{sect_quad_binary} and in MILP formulation of  Section~\ref{sect_f2l}; $w_{ij}^{(\ell)}=1$ if and only if there is a path of length at most $\ell$ from the vertex $i$ to the vertex $j$ in the tree.


\subsection{MIQP formulations}\label{sect_miqp}
In this section we consider three MIQP formulations for OCSTP. 

First in  Section~\ref{sect_brute} an enumerative algorithm is introduced, where all admissible tree topologies over internal vertices are sequentially enumerated, and QAP is solved for every specific topology to assign leaves. This brute-force approach stands apart from all other formulations studied in this article, since it reduces to a collection of MIQPs (one for each topology of internal vertices) while all other formulations reduce to a single MIP. Although brute-force algorithms were not treated seriously in the literature until the very early publications \citep{scott1969optimal}, the dramatic progress in software and hardware and, in particular, in parallel computing, makes us re-investigate this approach, also taking advantage of the recent advances in QAP algorithms and heuristics implemented in Gurobi solver. As mentioned above, any tailored algorithm should at least outperform this simple and easily parallelized routine. 

The second formulation presented in  Section~\ref{sect_LD} employs recent results from algebraic graph theory. It is shown in \cite{goubko2020bilinear} that a certain system of bilinear equations in $x_{ij}$ and $d_{ij}$ variables of a graph holds if and only if this graph is a tree. These equations can be used as equality constraints in many problems of tree topology design.

The third formulation (see  Section~\ref{sect_quad_binary}) uses binary variables that indicate if a path of at most certain length exists between a pair of vertices in a graph. Quadratic constraints model recursive relations between these variables in a tree.

%

\subsubsection{Enumeration of defoliated-trees and QAP}\label{sect_brute}

Let $\mathcal{T}(d)$ be the set of trees over the vertex set $V=\{1,...,n\}$ that share a (non-decreasing) vertex degree sequence $d=(d_1, ..., d_n)$ with $n_1$ \textit{leaves} (i.e., vertices with unit degree) and $m=n-n_1$ \emph{internal} vertices. If $T$ is a tree built over the set of internal vertices $\{n_1+1, ..., n\}$, and
\begin{equation}\label{eq_defoliated_degrees} 
d_T(i) \le d_i\text{ for all }i = n_1+1, ..., n,
\end{equation}
then $T$ gives rise to the collection of trees $\mathcal{T}_T\subset\mathcal{T}(d)$ obtained from $T$ by connecting $d_i - d_T(i)$ distinct leaves from the set $\{1, ..., n_1\}$ to every internal vertex $i = n_1+1, ..., n$. 

Since $\mathcal{T}(d)$ is a union of sets $\mathcal{T}_T$ for all such ``defoliated trees'' $T$, we have $\min_{G \in\mathcal{T}(d)} C_A(G) = \min_T \min_{G\in\mathcal{T}_T } C_A(G)$. For the given defoliated tree $T$ the distance matrix $D(T):=(d_{k\ell})_{k,\ell=1}^m$ is fixed, and the problem reduces to the optimal assignment of $n_1$ leaves to $m$ internal vertices, which is a standard QAP. Below $x_{ik}=1$ if a leaf $i$ is connected to a internal vertex $k$, and $x_{ik}=0$ otherwise.

\begin{problem}[F0Q]\label{MIQP0_f}
\begin{subequations}
\begin{align}
&\mbox{\emph{min}} \quad \sum\limits_{k,\ell=1:k\neq \ell}^m \left[2(d_{k\ell}+2)\sum\limits_{i,j=1:i<j}^{n_1}\mu_{ij}x_{ik}x_{jl} + 2(d_{k\ell}+1)\sum\limits_{i=1}^{n_1}\mu_{i,n_1+\ell}x_{ik}\right]\label{f0q_obj} \\
&\mbox{subject to} \nonumber\\
&\sum\limits_{k=1}^m x_{ik}=1 &\hspace{-10mm} \forall i=1,...,n_1,\label{f0q_xij_2}\\
&\sum\limits_{i=1}^{n_1}x_{ik}=d_k-d_T(k) &\hspace{-10mm} \forall k=n_1+1,...,n\label{f0q_xij_1}.
\end{align}
\end{subequations}
\end{problem}

According to the Cayley formula, the number of different defoliated trees grows superexponentially with the number of internal vertices $m$, however, for relatively small $m$ (namely, $m\le8$), probably the simplest approach to the optimal connecting tree problem is to enumerate all defoliated trees that satisfy Eq. (\ref{eq_defoliated_degrees}) and to find the best leaf assignment from the solution of the appropriate QAP. 
We used Pr{\"u}fer codes to economically enumerate all defoliated trees for the given degree sequence $d$, and the current record was used to cut off the computation of the next QAP in Gurobi.


\subsubsection{Formulation based on the bilinear matrix equation for trees}\label{sect_LD}

The next formulation is based on the matrix equation that involves Laplacian and distance matrix of a tree (Lemma 8.8 in \cite{bapat2010graphs}). Formally, if $T$ is a tree with the adjacency matrix $X = (x_{ij})_{i,j=1}^n$ and the distance matrix $D$, then the following matrix equation holds: 
\begin{equation}
\label{eq_LD}
\mathcal{L}D+2I=2(\mathbbm{1}-2d)\mathbbm{1}^T
\end{equation}
where $\mathcal{L}:=\diag d -X$ is the Laplacian matrix of the tree $T$, $d$ is its vertex degree sequence, $\diag d$ is the diagonal matrix with vertex degrees on its diagonal, while $I$ is an identity matrix and $\mathbbm{1}$ is an all-ones column vector of dimension $n$.

To illustrate the intuition behind this matrix equation and to write the corresponding MIQCP, we state and prove the following proposition:

\begin{proposition}\label{prop_LD_dir}
Let $T$ be a tree, then for any pair of vertices $i,j$ the following equality holds:
\begin{equation}
\label{distance_degree_eq}
\sum_{k\in N_T(i)} d_{kj}-d_id_{ij}=d_i-2.
\end{equation}
\end{proposition}
\begin{proof}
Although this fact follows immediately from the matrix equation \eqref{eq_LD}, we provide a simple and intuitive proof based on the observation that will be used further in this article. Let us consider a pair of non-adjacent vertices $i$, $j$, and observe that the shortest path in a tree from the vertex $i$ to $j$ must go through exactly one neighbor of the vertex $i$ (say, $q$), hence $d_{qj}=d_{ij}-1$. Moreover, the shortest paths from the other neighbors (if they exist) of the vertex $i$ to the vertex $j$ go through the vertex $i$, and distance from them to the vertex $j$ is $d_{ij}+1$, i.e., $d_{ik}=d_{ij}+1$ for any $k\in N_T(i)\setminus \{q\}$. 
Therefore, 
$$
\sum_{k\in N_T(i)} d_{kj}=d_{qj}+\sum_{k\in N_T(i)\setminus \{q\}} d_{kj}=d_{ij}-1+(d_i-1)(d_{ij}+1)=d_id_{ij}+d_i-2.
$$
In case of adjacent vertices $i$ and $j$, $q=j$, and we can use the same argumentation.
\end{proof}

By introducing explicitly distance variables $d_{ij}$ for all $i,j\in V$, we obtain the following formulation:

\begin{problem}[F1Q]
\label{MIQP1_f}
\begin{subequations}
\begin{align}
&\min\quad C_A(T)=
\sum\limits_{i,j=1:i\neq j}^n \mu_{ij} d_{ij} \label{f1q_obj} \\
&\mbox{subject to} \nonumber\\
&\sum\limits_{k\neq j:(i,k)\in E_s}x_{ik}d_{kj}-d_id_{ij}  =d_i-2 &\hspace{-10mm} \forall i,j\in V, i \neq j,\label{f1q_dij_2}\\
&\sum\limits_{j: (i,j)\in E_s} x_{ij} = d_i, &\forall i\in V,\\
&x_{ij} = x_{ji}\in \{0,1\}, &\forall (i,j)\in E_s,i<j,\label{f1q_symmx}\\
&d_{ij}=d_{ji} \in \{1,\dots, L\}, &\forall i,j\in V, i<j \label{f1q_dij_var}.
\end{align}
\end{subequations}
\end{problem}

This formulation uses $|V||V-1|$ integral variables, $|E_s|$ binary variables and $|V||V-1|$ bilinear constraints, which is not a lot for network design problems that involve pairwise distances between vertices.  

The validity of this formulation is justified in \cite{goubko2020bilinear}. In fact, the converse to Proposition \ref{prop_LD_dir} is proved: if the equation \eqref{eq_LD} holds for some graph with an adjacency matrix $X$ and a vertex degree sequence $d$, and for some matrix $D$, then this graph is a tree, and if $D$ has a zero diagonal, then $D$ is the distance matrix of this tree. Let us note that the constraints \eqref{f1q_dij_2} in formulation {\bf F1Q} are just the non-diagonal entries of the matrix equation \eqref{eq_LD} (according to  \cite{goubko2020bilinear}, the condition \eqref{f1q_dij_var} on variables $d_{ij}$ is enough to ensure that the obtained solution is a tree and $D$ is a distance matrix).

It is also shown in \cite{goubko2020bilinear} that if the matrix $D$ is assumed symmetric and the matrix equation (\ref{eq_LD}) is limited to its upper triangular part, the same conclusion about the arboricity of the graph still can be made. This means that the formulation {\bf F1Q} is also valid if equations \eqref{f1q_dij_2} for $i>j$ are removed and all occurences of distance variables $d_{ij}$ for $i>j$ are replaced with $d_{ji}$. 

In computational experiments below we denote this refined formulation by {\bf F1Q$^\urcorner$}. Although {\bf F1Q$^\urcorner$} has fewer variables and constraints, {\bf F1Q} might still have a tighter relaxation and so, is worth to consider.

Although equality constraints in (\ref{f1q_dij_2}) involving binary variables $x_{ik}$ are non-convex, they can be converted into linear constraints \citep{GurobiMIBP} by introducing one auxiliary integral variable and adding $4$ linear inequality constraints for each of $\Theta(|E_s||V|)$ independent bilinear terms $x_{ik}d_{kj}$. The problem thus becomes a MILP, which is efficiently handled by the modern optimization software like Gurobi. 

Finally, let us note that this formulation easily extends to the general OCSTP where graph edges have arbitrary lengths $t_{ij}$, $(i,j)\in E_s$ by replacing the equality constraint \eqref{f1q_dij_2} with its weighted version from \cite{goubko2020bilinear}  
\begin{equation*}
\sum\limits_{k\neq j:(i,k)\in E_s}\frac{x_{ik}}{t_{ik}}d_{kj}-\sum\limits_{k\neq j:(i,k)\in E_s}\frac{x_{ik}}{t_{ik}}d_{ij} =d_i-2, \hspace{10pt} \forall i,j\in V, i \neq j,\tag{\ref{f1q_dij_2}$'$}
\end{equation*}
and relaxing the integrality constraint \eqref{f1q_dij_var}:
\begin{equation*}
d_{ij}=d_{ji} \ge t_{ij}, \hspace{10pt} \forall i,j\in V, i<j. \tag{\ref{f1q_dij_var}$'$}
\end{equation*}

\subsubsection{Bilinear distance-based formulation with binary distance variables}\label{sect_quad_binary}

The next MIQCP formulation uses binary variables $w_{ij}^{(\ell)}$ that indicate whether a path of length at most $\ell$ exists from the vertex $i$ to the vertex $j$ and employs recursive constraints to ensure that these variables describe distances in a tree. A similar technique has proved its efficiency in graph fragmentation \citep{veremyev2015critical,veremyev2019finding} and subgraph detection \citep{Veremyev11,matsypura2019exact,kim2020maximum} problems that require modeling of pairwise distances between graph vertices. In \cite{mukherjee2017minimum,diaz2019robust} this approach has also been applied to the problems of network design. However, only MILP formulations were studied in the literature. Below we introduce an MIQCP version of this approach with bilinear constraints.

Thus, $w^{(\ell)}_{ij}=1$ if there is a path of length at most $\ell$ from vertex $i$ to $j$ in a tree $T$, and $w^{(\ell)}_{ij}=0$ otherwise, $i,j=1,\dots,n,\, i\neq j; \ell=1,\dots,L$, where $L=m+2$ is the maximum possible diameter of a spanning tree with the vertex degree sequence $d$ for the graph $G=\langle V, E_s\rangle$.
It is clear that $\sum_{\ell=1}^L\left(1-w^{(\ell)}_{ij}\right)=d_{ij}(T)-1$. Taking into account that the path of the length $L$ always exists, and so $w^{(L)}_{ij}\equiv 1$, we deduce that $d_{ij}(T) = L-\sum_{\ell=1}^{L-1}w^{(\ell)}_{ij}$, and the cost function (\ref{eq_cost}) of the ORSTP for the tree $T$ can be written as

\begin{equation}
C_A(T)=
\sum\limits_{i,j=1:i\neq j}^n \mu_{ij}d_{ij}(T)=\sum\limits_{i,j=1:i\neq j}^n \mu_{ij}\left(L-\sum\limits_{\ell=1}^{L-1} w^{(\ell)}_{ij}\right),
\end{equation}

Then, the constrained ORSTP can be written as: 

\begin{problem}[F2Q]
\label{MIQP2_f}
\begin{subequations}
\begin{align}
&\min \quad C_A(T)=
\sum\limits_{i,j=1:i\neq j}^n \mu_{ij}\left(L-\sum\limits_{\ell=1}^{L-1} w^{(\ell)}_{ij}\right) \label{f2q_obj} \\
&\mbox{subject to} \nonumber\\
& w^{(1)}_{ij}= x_{ij}, 
&\hspace{-40mm} \forall(i,j)\in E_s\label{f2q_w1_1}\\
& w^{(1)}_{ij}= 0, &\hspace{-40mm} \forall(i,j) \notin E_s, i< j \label{f2q_w1_2}\\
& w^{(\ell)}_{ij}\leq x_{ij}+\sum\limits_{k\neq j:(i,k)\in E_s}x_{ik}w^{(\ell-1)}_{kj}, &\hspace{-30mm}\forall (i,j)\in E_s, \ i< j, \ell\in\{2,\dots,L\}\label{f2q_wl_1}\\
& w^{(\ell)}_{ij}\leq \sum\limits_{k\neq j:(i,k)\in E_s}x_{ik}w^{(\ell-1)}_{kj}, &\hspace{-40mm}\forall (i,j)\notin E_s, i< j,\ \ell\in\{2,\dots,L\}\label{f2q_wl_2}\\
& w^{(L)}_{ij}=1 \label{f2q_wL}, &\hspace{-40mm}\forall i,j\in V, i< j\\
&\sum\limits_{j: (i,j)\in E_s} x_{ij} = d_i, &\forall i\in V\label{f2q_deg}\\
&x_{ij}=x_{ji} \in \{0,1\}, &\forall (i,j)\in E_s, i<j\\
&w^{(\ell)}_{ij}=w^{(\ell)}_{ji}\ \in \{0,1\}, & \hspace{-70mm} \forall i,j\in V, i< j, \ell\in\{1,\dots,L\}.
\end{align}
\end{subequations}
\end{problem}

In contrast to {\bf F1Q}, this formulation involves $\Theta(L|V|^2)$ binary variables and $\Theta(L|V|^2)$ bilinear constraints. Constraints \eqref{f2q_w1_1}-\eqref{f2q_wL} recursively model distances in the tree $T$. First, constraints \eqref{f2q_w1_1}-\eqref{f2q_w1_2} ensure that if an edge $(i,j)\in E_s$ is present in the tree $T$ ($x_{ij}=1$), then the distance variable $w^{(1)}_{ij}=1$, otherwise $w^{(1)}_{ij}=0$. Then, constraints \eqref{f2q_wl_1}-\eqref{f2q_wl_2} aim at ensuring that if there is no path of length at most $\ell-1$ between any neighbor of vertex $i$ and vertex $j$, then, unless an edge $(i,j)$ is not in a tree, there is no path of length $\ell$ between vertices $i$ and $i$ in the selected tree $T$ (i.e., $w^{(\ell)}_{ij}$ is forced to be 0). Constraints \eqref{f2q_wL} guarantee that any pair of vertices in the selected tree $T$ is connected and have the distance at most $L$. They can also be used to additionally limit the diameter of admissible trees. 

Note that the validity of this formulation can be easily verified due to its intuitive interpretation of variables and constraints. Moreover, the integrality of variables $w^{(\ell)}_{ij}$ can be relaxed. Below we provide a formal proof. 

\begin{proposition}
\label{prop_validity_F2Q}
There exists an optimal solution {\bf x$^*$, w$^*$} of the problem {\bf F2Q} with relaxed variables {\bf w$^*$} \emph{(}i.e., $w^{(\ell)}_{ij}=w^{(\ell)}_{ji} \in [0,1]$ $\forall i,j\in V, i< j, \ell\in\{1,\dots,L\}$\emph{)}, such that variables ${\bf x^*}$ define an optimal tree $T^*$ and the objective value of {\bf F2Q} is equal to its cost $C_A(T^*)$.
\end{proposition}

\begin{proof}
The problem {\bf F2Q} is feasible since for any admissible spanning tree $T=\langle V,E\rangle$ one can construct a feasible solution {\bf x, w} by setting $x_{ij}=1$ for all $(i,j)\in E$ and $w_{ij}^{(\ell)}=w_{ji}^{(\ell)}=1$ for all $i,j\in V, i<j, \ell\ge d_T(i,j)$, and filling the rest of variables to zero.

Let us consider an optimal solution {\bf x$^*$, w$^*$}, and let $T^*$ be a subgraph of $G=\langle V, E_s\rangle$ that contains edges $(i,j)\in E_s$ for which $x^*_{ij}=1$. We first show that $T^*$ is a tree. Since the vertex degree sequence $d$ is arborescent, constraints \eqref{f2q_deg} assure that $T^*$ has exactly $n-1$ edges, and we only need to show that subgraph $T^*$ is connected. 

We prove even a stronger fact that if $w^{*(\ell)}_{ij}>0$, then the distance between nodes $i$ and $j$ in $T^*$ is at most $\ell$ (i.e., $d_{ij}(T^*)\leq \ell$). This can be done by induction on $\ell$. For the base of induction, $\ell=1$, this fact is obvious due to constraints \eqref{f2q_w1_1}. Now assume that the statement is true for some $\ell=\tau$, $\tau\in\{2,\dots,L\}$, i.e, if  $w^{*(\tau)}_{ij}>0$, then $d_{ij}(T^*)\leq \tau$. Let us prove that it is also true for $\ell=\tau+1$. 

Suppose that there exists a pair of nodes $a<b\in V$ such that $w^{*(\tau+1)}_{ab}=w^{*(\tau+1)}_{ba}>0$. We need to show that $d_{ab}(T^*)\leq \tau+1$. Let us consider the case when $(a,b)\in E_s$. Note that the constraint \eqref{f2q_wl_1} implies that
$x_{ab}+\sum\limits_{k\neq b:(a,k)\in E_s}x^*_{ak}w^{(\tau)}_{kb}>0$, which means that $x^*_{ab}=1$, or there exists a node $k\in V$, such that $x^*_{ak}=1$ and $w_{kb}^{*(\tau)}>0$. By the induction assumption it means that either $d_{ab}(T^*)=1\leq \tau$ or  there exists a neighbor $k$ of node $i$ in $T^*$, such that $w_{kb}^{*(\tau)}>0$ implying that $d_{kb}(T^*)\leq \tau$, hence $d_{ab}(T^*)\leq \tau+1$.
Similarly, for $(a,b)\notin E_s$ we can use the constraint \eqref{f2q_wl_2} and also conclude that $d_{ab}(T^*)\leq \tau+1$. Therefore, from identities \eqref{f2q_wL} it follows that $d_{ij}(T^*)\leq L$ for any pair of nodes $i,j\in V$ and $T^*$ is connected (and so, is a tree).

Since we proved that $T^*$ is a tree, let us define {\bf w$'$} = $\{w'^{(\ell)}_{ij}:\forall i\neq j\in V, \ell\in\{1,\dots,L\}\}$ as follows. For any pair $i\neq j\in V$:
\begin{equation}
w'^{(\ell)}_{ij}=
\begin{cases}
1, \text{ if } d_{ij}(T^*)\leq \ell, \\
0, \text{ if } d_{ij}(T^*)> \ell. 
\end{cases}
\end{equation}
One can easily verify that {\bf x$^*$, w$'$} is a feasible solution of {\bf F2Q} since it satisfies all the constraints. In addition, since we proved that $w^{*(\ell)}_{ij}>0$ implies that $d_{ij}(T^*)\leq \ell$, we can conclude also that $w^{*(\ell)}_{ij}\leq w'^{(\ell)}_{ij}$,  $\forall i,j\in V, i\neq j, \ell\in\{1,\dots,L\}$.

Moreover, by definition of {\bf w$'$} it follows that $C_A(T^*)=
\sum\limits_{i,j=1:i\neq j}^n \mu_{ij}\left(L-\sum\limits_{\ell=1}^{L-1} w'^{(\ell)}_{ij}\right)$ and, since the objective function \eqref{f2q_obj} is monotone, 
$$\sum\limits_{i,j=1:i\neq j}^n \mu_{ij}\left(L-\sum\limits_{\ell=1}^{L-1} w'^{(\ell)}_{ij}\right)\leq \sum\limits_{i,j=1:i\neq j}^n \mu_{ij}\left(L-\sum\limits_{\ell=1}^{L-1} w^{*(\ell)}_{ij}\right).$$ 
Since, by assumption, {\bf x$^*$, w$^*$} is the optimal solution of the problem {\bf F2Q}, this inequality becomes an equality, and the objective of {\bf F2Q} gives exactly the optimal cost $C_A(T^*)$. 

Hence, we showed that an optimal solution of {\bf F2Q} corresponds to a tree and its objective value is the communication cost \eqref{eq_cost} of this tree. To verify that this tree is optimal, i.e., has the minimum communication cost among all admissible trees, one can easily observe that for any admissible tree we can construct the feasible solution of {\bf F2Q} such that its objective value is equal to the communication cost of that tree, which completes the proof.
\end{proof}

This proof uses monotonicity of the cost function \eqref{eq_cost} with respect to $w_{ij}^{(\ell)}$. At the same time, by introducing additional inequality constraints this formulation can be correctly extended to non-monotone cost functions, which, however, goes beyond the scope of the present article.

Similarly to {\bf F1Q}, the formulation {\bf F2Q} can be converted into MILP by introducing an auxiliary variable and $4$ additional linear constraints for each bilinear term $x_{ik}w_{kj}^{(l-1)}$ in (\ref{f2q_wl_1}) and (\ref{f2q_wl_2}). This transformation is performed automatically by Gurobi 8.0 and higher \citep{GurobiMIBP}.

At first sight, {\bf F2Q} may seem a reformulation of {\bf F1Q} using binary variables for distances. However, it grounds on a different idea. To see the difference, please note that {\bf F2Q} easily generalizes to graphs with loops, where the distance between vertices becomes the shortest-path distance. On the other hand, as conjectured in \cite{goubko2020bilinear}, the formulation {\bf F1Q}, when generalized to graphs with loops, resembles another concept of distances in a graph, known as the \textit{resistance distance}. 

\subsection{Linear MIP formulations}\label{sect_milp}

\subsubsection{Flow-based formulation}\label{sect_flow}



Appeared in 1962, more than half a century ago, the seminal book ``Flows in networks'' by L.\,R~Ford and D.\,R.~Fulkerson \citep{fordflows} laid in the ground of the new branch of discrete mathematics, the \textit{network optimization}. The algorithms presented in this book belong to the classics of combinatorial optimization. The magic of elegant and efficient models of network flows motivated researchers to use their recipes in many problems of network optimization. Since the minimum length of the path traversed by a commodity flow between two vertices in a network is equal to the shortest-path distance between these vertices, the multicommodity flow-based formulation becomes a natural modeling tool for many distance-related problems of network optimization. So, it is no wonder that it remains the most popular MILP formulation for OCSTP \citep{fischetti2002exact,zetina2019exact}, for the hop-constrained spanning tree \citep{gouveia1996multicommodity,Pirkul03,Botton11}, and for many other problems of network analysis and design.

Although many variations of a flow-based formulation exist (see Section~\ref{sect_review}), here we present the simplest one, which reflects its essence. Let $u^{st}_{ij}\in [0,1]$ for all $s,t\in V$ and $(i,j)\in E_s$ denote the amount of the flow sent in the graph $G=\langle V,E_s\rangle$ from the vertex $s\in V$ to the vertex $t\in V$ through an edge $(i,j)\in E_s$, $s<t$. Assuming the unit flow has to be sent from the vertex $s\in V$ to the vertex $t\in V$ in a tree $T$, we have 
$$\sum\limits_{(i,j)\in E_s} u^{st}_{ij}\geq d_{st}(T),$$ 
and a flow always exists from $s$ to $t$  that reduces this inequality to the equality (this flow goes through the shortest path). So, the constrained OCSTP can be written as the following multicommodity flow planning problem:  

\begin{problem}[F0L]
\begin{subequations}
\begin{align}
&\min\quad C_A(T)=
\sum\limits_{s,t=1: s\neq t}^n \mu_{st} \sum\limits_{(i,j)\in E_s} u^{st}_{ij} \label{flow_obj} \\
&\quad\mbox{subject to} \nonumber\\
&\sum\limits_{j:(s,j) \in E_s}  u^{st}_{sj}-\sum\limits_{i:(i,s) \in E_s}  u^{st}_{is} \geq 1 & \forall s,t\in V,\ s<t,\label{flow_1}\\
&\sum\limits_{i:(i,t) \in E_s}  u^{st}_{it}-\sum\limits_{j:(t,j) \in E_s}  u^{st}_{tj}\geq 1 & \forall s,t\in V,\ s<t,\label{flow_2}\\
&\sum\limits_{j: (i,j) \in E_s}  \left(u^{st}_{ij}-u^{st}_{ji}\right)=0 & \forall s,t\in V, \ s<t, \ \forall i\in V\setminus\{s,t\},\label{spanner_flow_3}\\
& u^{st}_{ij}\leq   x_{ij} &\forall s,t\in V,\ s<t, \ \forall (i,j) \in E_s, \label{flow_4}\\
&\sum\limits_{j: (i,j)\in E_s} x_{ij} = d_i, &\forall i\in V\\
& x_{ij} = x_{ji} \in \{0,1\}, \ 0\leq  u^{st}_{ij}\leq 1 &\forall s,t\in V, \  s<t,\ \forall (i,j)\in E_s.
\end{align}
\end{subequations}
\end{problem}

The constraints \eqref{flow_1}-\eqref{flow_4} are the standard flow conservation conditions. The flow-based formulation requires $\Theta(|E_s|)$ binary variables,  $\Theta(|R||E_s|)$ continuous variables and $\Theta(|R||E_s|)$ linear constraints; hence, in case of the dense requirements matrix and the graph (so that $|R|=\Theta(V^2)$) and $|E_s|=\Theta(V^2)$), the formulation {\bf F0L} may become intractable even for relatively small $n$. 

The validity of {\bf F0L} has been many times established in the literature \citep{fischetti2002exact,zetina2019exact}. Here we just note that the proofs typically take use of the fact that the cost function is monotone in distances $d_{ij}(T)$. Although for tree search problems, where the sole path connects any pair of vertices in any target tree, {\bf F0L}  can be extended to non-monotone cost functions by introducing the ``no return to vertex'' conditions $\sum_{j:(i,j)\in E_s}u_{ij}^{st}\le 1, \sum_{j:(i,j)\in E_s}u_{ji}^{st}\le 1$ for all $i\in V, (s,t)\in R$, this trick does not help in general network design problems.  

\subsubsection{Distance-based formulation with integral distance variables}\label{sect_f1l}

If there is an edge $(i,j)\in E$ in a connected graph $T=\langle V, E\rangle$, the distance $d_{ij}(T)=1$. Otherwise the shortest path from the vertex $i$ to the vertex $j$ has to go through some neighbor $k$ of the vertex $i$ in $T$. Hence, the distance $d_{ij}(T)$ exceeds exactly by one the distance $d_{kj}(T)$. Moreover, for any other neighbor $k'$ of the vertex $i$ in $T$ the distance $d_{ij}(T)\le 1+d_{ik'}(T)$ (or else the shortest path to the vertex $j$ would traverse the vertex $k'$). Therefore,
\begin{equation}\label{eq_bellman}
d_{ij}(T)=\begin{cases}
 1                          & \text{if }(i,j)\in E,\\
 1+\min_{k:(i,k)\in E} d_{ik}(T)   & \text{if }(i,j)\notin E. 
 \end{cases} \hspace{10mm}\forall i,j\in V, i\neq j  
\end{equation}
Surprisingly, this system of equations determines recursively the whole distance matrix $D(T)=(d_{ij}(T))_{i,j=1}^n$ of any connected graph $T$ (see Proposition \ref{prop_f1l} below).

The two MILP formulations presented in the following sections take use of this idea. Similarly to {\bf F1Q}, the first formulation introduced in the  present section employs integral variables $d_{ij}$ to model distances $d_{ij}(T)$ in a target tree $T$. 

Following the big-$M$ reformulation approach, the non-linear equations \eqref{eq_bellman} are transformed into linear inequalities using auxiliary binary variables  
$y_{ikj}$ introduced for all $i,j=1,\dots,n;\, i\neq j \neq k; (i,k)\in E_s$, such that $y_{ikj}=1$ if the shortest path from the vertex $i$ to the vertex $j$ goes through the neighbor $k$ of the vertex $i$ in the target tree $T$, and $y_{ikj}=0$ otherwise.

\begin{problem}[F1L]
\label{MILP1_f}
\begin{subequations}
\begin{align}
&\min \quad C_A(T)=
\sum\limits_{i,j=1:i\neq j}^n \mu_{ij} d_{ij} \label{f1l_obj} \\
&\mbox{subject to} \nonumber\\
& x_{ij}= x_{ji}&\hspace{-10mm} \forall (i,j)\in E_s, i< j\label{dij_1}\\
& d_{ij} \geq d_{kj}+1-M(1-y_{ikj})&\hspace{-10mm} \forall i,j\in V, i < j, (i,k)\in E_s, i< j\label{f1l_dij_2}\\
&\sum\limits_{k: (i,k)\in E_s} y_{ikj}=1-x_{ij},  &\forall(i,j)\in E_s,  i< j \label{f1l_y1a} \\
&\sum\limits_{k: (i,k)\in E_s} y_{ikj}= 1, & \forall (i,j)\notin E_s, i< j \label{f1l_y1b} \\
& y_{ikj}\leq x_{ik},  & \hspace{-10mm} \forall i<j\in V, (i,k)\in E_s, i\neq j \neq k, \label{f1l_y2}\\
&\sum\limits_{j: (i,j)\in E_s} x_{ij} = d_i, &\forall i\in V\\
&x_{ij} \in \{0,1\}, &\forall (i,j)\in E_s\\
&d_{ij}=d_{ji} \in \{1,...,L\}, &\forall (i,j)\in E_s, i< j\label{f1l_dij_lim}\\
&y_{ikj} \in \{0,1\}, & \hspace{-10mm} \forall i<j\in V,\ (i,k)\in E_s, i\neq j \neq k.
\end{align}
\end{subequations}
\end{problem}

In the formulation above, $M$ is a constant large enough to make the corresponding constraint inactive when needed, which should be at least the maximum possible diameter $L$ of the target tree. Note that although the intuition behind this formulation is quite simple, its validity is not obvious. Indeed, one may easily verify that for any tree $T$ with the vertex degree sequence $d$ there exists a feasible solution {\bf x, y, d}, satisfying constraints of formulation {\bf F1L} such that the values of variables {\bf x, y, d} correspond to the edge set, the set of neighbors laying on shortest paths, and distances between vertices in that tree $T$, respectively. However, it is not obvious why the optimal solution {\bf x$^*$, y$^*$, d$^*$} should give a connected graph and, moreover, the optimum communication spanning tree. Below we prove that in any optimal solution  {\bf x$^*$, y$^*$, d$^*$} of {\bf F1L} the values of {\bf x$^*$} define a connected tree and {\bf d$^*$} give exactly its pairwise distances.

\begin{proposition}\label{prop_f1l}
Let {\bf x$^*$, y$^*$, d$^*$} be an optimal solution of problem {\bf F1L}. Then the values of ${\bf x^*}$ define an optimal tree and the objective value gives its communication cost.
\end{proposition}

\begin{proof}
Note that since for the optimal tree $T$ there exists the feasible solution {\bf x, y, d}, satisfying constraints of formulation {\bf F1L} such that the values of variables {\bf x, y, d} correspond to the edge set, the set of neighbors laying on shortest paths, and distances between vertices in that tree $T$, respectively, and the objective value being equal to the optimal spanning tree communication cost, then the optimal tree belongs to the feasible set of {\bf F1L}, and the optimal solution of {\bf F1L} has finite values of all variables.

First, we show that the values of ${\bf x^*}$ define a connected graph. Assume that the variables {\bf x$^*$} in the optimal solution of {\bf F1L} define a disconnected graph (i.e., a graph containing edges $(i,j)\in E_s$ for which $x^*_{ij}=1$). Hence, it has at least two non-empty connected components $C_1\subset V$ and $C_2\subset V$. Since all variables take finite values in this solution, let $d^*_{ij}=\argmin\{d^*_{ab}: a\in C_1, b\in C_2\}$  be the minimum value of $d^*_{ab}$ between vertices $a\in C_1$ and $b\in C_2$ in two connected components. Assume that $i<j$ (the case when $i>j$ is considered similarly). Then, due to constraints \eqref{f1l_dij_2} - \eqref{f1l_y2}, there exists exactly one neighbor $k$ of $i$ (hence, $k\in C_1$), such that $d^*_{ij}\geq d^*_{kj}+1$, or, $d^*_{kj}\leq d^*_{ij}-1$, which contradicts the assumption. 

Hence, ${\bf x^*}$ defines a tree $T^*$. Let ${\bf x^*, y', d'}$ be a feasible solution, which satisfies the constraints of formulation {\bf F1L}, such that the values of variables ${\bf x^*, y', d'}$ correspond to the edge set, the set of neighbors laying on the shortest paths, and the distances between vertices in that tree $T$, respectively. We will show that $d^*_{ij}\geq d'_{ij}$ for any $i<j\in V$.

Assume that there exists a pair of non-adjacent vertices $a,b\in V$, such that $d^*_{ij}< d'_{ij}$ (which is not possible for adjacent vertices by definition of variables {\bf d}). Define $d^*_{ij}:=\argmin\{d^*_{ab}: d^*_{ij}< d'_{ij}\}$, so that $i,j$ is a pair of vertices with the minimum possible value of $d^*_{ij}$ such that that the distance between $i$ and $j$ in a tree $T^*$ is greater than $d^*_{ij}$. Then, due to the constraints \eqref{f1l_dij_2} - \eqref{f1l_y2}, there exists exactly one neighbor $k$ of $i$, such that $d^*_{ij}\geq d^*_{kj}+1$. Hence, $d^*_{kj}\leq d^*_{ij}-1$, and, therefore $d^*_{kj}\geq d'_{kj}$ (by the choice of $i,j$). Moreover, $d'_{kj}\leq d'_{ij}+1$. Then, we can conclude that
$$
d^*_{ij}\geq d^*_{kj}+1 \geq d'_{kj}+1\geq d'_{ij}
$$
which contradicts the assumption on the selection of $d^*_{ij}$. Hence, $d^*_{ij}\geq d'_{ij}$ for any $i<j\in V$ and since ${\bf d^*}$ is an optimal solution of {\bf F1L} and its objective value function is non-decreasing function of ${\bf d}$, then ${\bf x^*}$ defines the optimal communication tree and the objective value corresponds to its optimal communication cost.
\end{proof}

This formulation involves $\Theta(|V|^2)$ integral variables, $\Theta(|V||E_s|)$ binary variables, and $\Theta(|V||E_s|)$ linear constraints. Moreover, integral variables $d_{ij}$ can be relaxed. However, this relaxation gives no considerable performance improvement in our computational experiments, therefore, we keep them integral.

As a final remark, we note that this formulation easily generalizes to the case of the general OCSTP, where graph edges have different lengths $t_{ik}$, $(i,k)\in E_s$. In fact, it is enough to adjust the inequalities \eqref{f1l_dij_2} by replacing the constant $1$ with the edge length $t_{ik}$  to obtain 
\begin{equation*}
d_{ij} \geq d_{kj}+t_{ik}-M(1-y_{ikj}) \hspace{10pt} \forall i,j\in V, i < j, (i,k)\in E_s, i< j. \tag{\ref{f1l_dij_2}$'$}    
\end{equation*}
and by relaxing the integrality constraints \eqref{f1l_dij_lim}.


\subsubsection{Distance-based formulation with binary distance variables}\label{sect_f2l}

The next formulation is inspired by by the MILP model from  \cite{mukherjee2017minimum,diaz2019robust}. Being a linearization of {\bf F2Q}, it has a very similar set of recursive constraints. However, some constraints of the original model from \cite{mukherjee2017minimum,diaz2019robust}, when applied to OCSTP, become redundant due to the monotonicity of the cost function in distance variables (see Section~\ref{sect_quad_binary} for details).


For more details on modeling graph distances with binary variables (which usually can be relaxed) and linear constraints in related problems and corresponding MIP formulations we refer a reader to \cite{veremyev2015critical,diaz2019robust,matsypura2019exact,veremyev2019finding}.   

Since in {\bf F2Q} each term $x_{ik}w^{(\ell)}_{kj}$  for all $i,j=1,\dots,n;\, i<j \neq k; \ell=2,\dots,L, (i,k)\in E_s $ needs to be linearized, we intruduce extra variables $y^{(\ell)}_{ikj}$ for all $i,j=1,\dots,n;\, i< j \neq k; \ell=2,\dots,L, (i,k)\in E_s $, such that    $y^{(\ell)}_{ikj}=x_{ik}w^{(\ell)}_{ij}$. Then, we can write the following linearized version of {\bf F2Q}:

\begin{problem}[F2L]
\label{MILP2_f}
\begin{subequations}
\begin{align}
&\min \quad C_A(T)=
\sum\limits_{i,j=1:i\neq j}^n \mu_{ij}\left(L-\sum\limits_{\ell=1}^{L-1} w^{(\ell)}_{ij}\right) \label{f2l_obj} \\
&\mbox{subject to} \nonumber\\
& w^{(1)}_{ij}= x_{ij}, &\hspace{-40mm} \forall(i,j)\in E_s\label{f2l_w1_1}\\
& w^{(1)}_{ij}= 0, &\hspace{-40mm} \forall(i,j) \notin E_s, i< j\\
& w^{(\ell)}_{ij}\leq x_{ij}+\sum\limits_{k\neq j:(i,k)\in E_s}y^{(\ell)}_{ikj}, &\hspace{-40mm}\forall (i,j)\in E_s, \ i< j, \ell\in\{2,\dots,L\}\\
& w^{(\ell)}_{ij}\leq \sum\limits_{k\neq j:(i,k)\in E_s}y^{(\ell)}_{ikj}, &\hspace{-40mm}\forall (i,j)\notin E_s, i< j,\ \ell\in\{2,\dots,L\}\\
& w^{(L)}_{ij}=1, &\hspace{-40mm}\forall i,j\in V, i< j\\
&\sum\limits_{j: (i,j)\in E_s} x_{ij} = d_i, &\forall i\in V\\
& y^{(\ell)}_{ikj}\leq x_{ik},\ y^{(\ell)}_{ikj}\leq w^{(\ell-1)}_{kj},  & \hspace{-50mm} \forall i<j\in V, (i,k)\in E_s, j \neq k, \ell\in\{2,\dots,L\}\label{f2l_y_l}\\
&x_{ij}= x_{ji}\in \{0,1\}, &\forall (i,j)\in E_s, i<j\\
&w^{(\ell)}_{ij}=w^{(\ell)}_{ji}, & \hspace{-50mm} \forall i,j\in V: i< j, \ell\in\{1,\dots,L\}\\
&y^{(\ell)}_{ikj} \in \{0,1\}, & \hspace{-50mm} \forall (i,k)\in E_s, j\in V: i< j \neq k, \ell\in\{1,\dots,L\}.
\end{align}
\end{subequations}
\end{problem}
 
This formulation contains  $\Theta(L|V||E_s|)$ binary variables and $\Theta(L|V||E_s|)$ linear constraints, which is larger than the linear formulation {\bf F1L}, but contains only binary variables. The validity of this formulation is easily verified using the similar arguments as in the proof of Proposition \ref{prop_validity_F2Q}. Also, it should be noted that the variables {$\bf w$, $\bf y$} can be relaxed as well.

Being a linearization of {\bf F2Q}, formulation {\bf F2L} also hardly generalizes beyond ORSTP, since the distance is measured as an edge count in a shortest path. At the same time, an original model with binary distance variables has been generalized to integral edge weights in a quasi-polynomial manner in \cite{veremyev2015critical}.




\section{Computational Experiments}\label{sect_comp}

\subsection{Environment}

For our computational experiments we generated a series of requirements matrices of different size. They are random principal submatrices of eight distinct origin-destination matrices  (OD-matrices) from transportation industry introduced in \cite{goubko2020lower} (see Table \ref{tb_datasets}). More details on these data sets can be found in \cite{goubko2020lower}. The pattern of requirements in every matrix was generated to form a connected structure of traffic flows.

\begin{table}[ht]
\scriptsize
\setlength{\tabcolsep}{3pt}
\renewcommand{\arraystretch}{1.2}
\caption{Data sets used to generate requirements matrices for computational experiments}
\begin{center}
\begin{tabular}{c|c|c|l}
\hline
No & Abbreviation & Original size & Description \\
\hline
1 &	BCC & 20 & Queensland TransLink ``BCC Ferry'' ferry service passenger flows \\
2 &	Sunbus & 846 & Queensland TransLink ``Sunbus'' local bus trips \\
3 &	PRT & 498 & Queensland TransLink ``Park Ridge Transit'' express bus trips \\
4 &	MGBS & 364 & Queensland TransLink ``Mt Gravatt Bus Service'' carrier trips \\
5 &	QR & 154 & Queensland TransLink ``Queensland Rail'' urban train passenger flows \\
6 &	ANPR & 91 & Greater Cambridge vehicle journeys (ANPR Data) \\
7 &	US & 269 & The Air Carrier flight statistics for ``US Airways Inc.''\\
8 &	TW & 176 & The Air Carrier flight statistics for ``Trans World Airways LLC''\\
\hline
\end{tabular}
\end{center}
\label{tb_datasets}
\end{table}%

The eight source matrices represent diverse internal structure of traffic flows. They differ in the density and in the pattern of flow concentration (e.g., presence of popular destinations or busy routes). Several examples of generated OD-submatrices over 15 vertices are shown in Figure~\ref{fig_sample_flow}. The tone of the cell $(i,j)$ (from white to black) corresponds to the more intensive traffic flow (requirement) from the source $i$ to the destination $j$.

\begin{figure}[th]
\centering
\subfloat[Submatrix of {\bf BCC} OD-matrix: density=0.88]{\label{2_15_1}\includegraphics[scale=0.33]{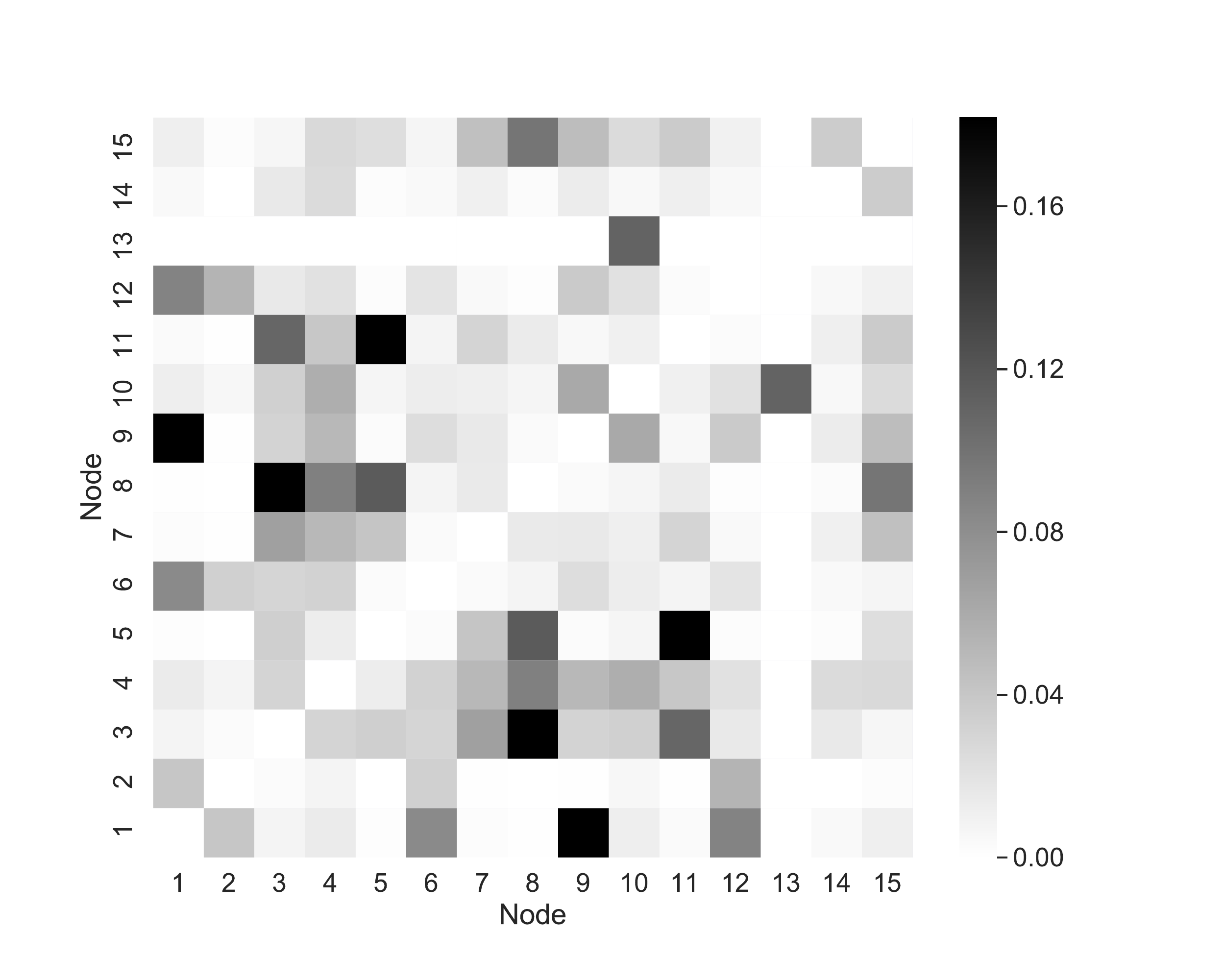}} 
\subfloat[Submatrix of {\bf Sunbus} OD-matrix: density=0.65]{\label{3_15_1}\includegraphics[scale=0.33]{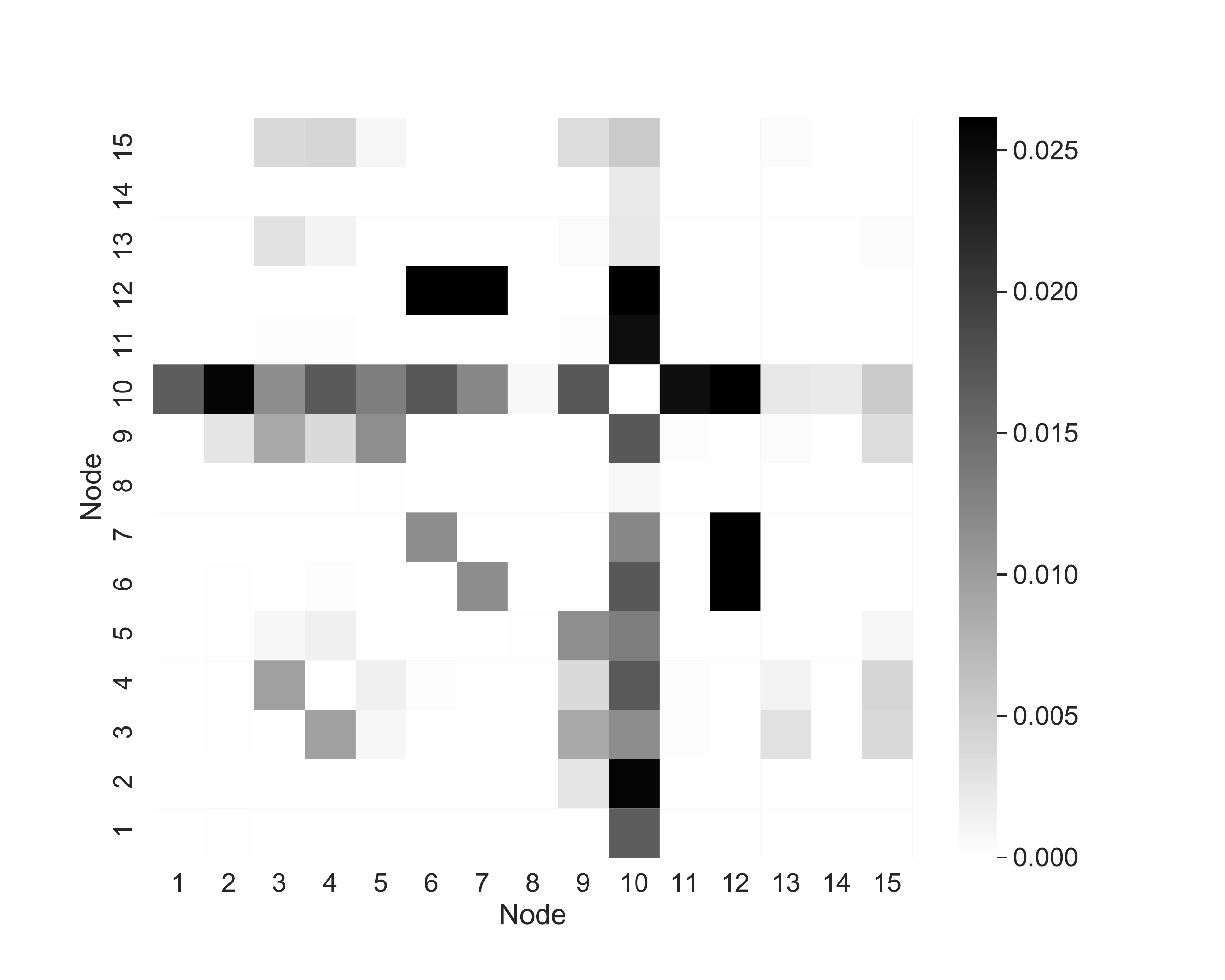}}\\
\subfloat[Submatrix of {\bf QR} OD-matrix: density=0.16]{\label{6_15_1}\includegraphics[scale=0.33]{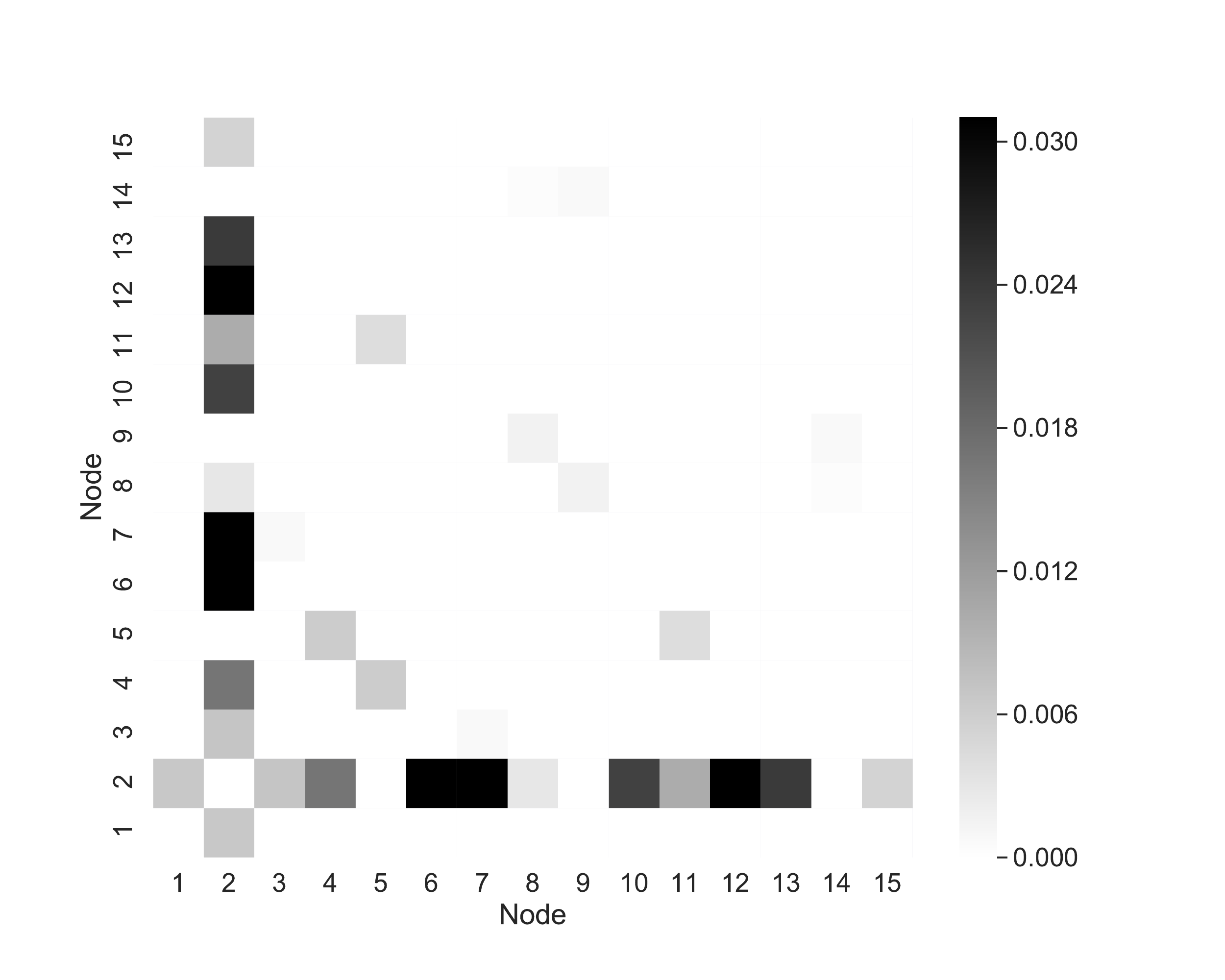}}
\caption{Example of three 15x15 requirements matrices used in computational experiments.}
\label{fig_sample_flow}
\end{figure}

For every problem dimension (order of a spanning tree) $n$, ten random degree sequences were generated with internal vertex degrees uniformly distributed from 2 to 5, and for a given problem dimension $n$ the constrained OCSTP was solved for every requirements matrix of order $n$ against all ten degree sequences of order $n$. Diversity of vertex degree sequences is important because a degree sequence may dramatically affect the volume of the search space and, as a consequence, the computation time of branch-and-bound algorithms.  

In our experiments we compared three MIQP formulations and three MILP formulations introduced in the previous section (see the list and a shorthand notation in Table \ref{tb_formulations}). Five of these six formulations are novel, with the only exception of \textbf{F0L}.

\begin{table}[ht]
\scriptsize
\setlength{\tabcolsep}{3pt}
\renewcommand{\arraystretch}{1.2}
\caption{MIP formulations compared in computational experiments}
\begin{center}
\begin{tabular}{  m{1cm} | m{9cm}| m{1cm} | m{4cm}}
\hline
Abb. & Name & Section & Brief comment \\
\hline
\multicolumn{4}{c}{Mixed-integer quadratic formulations}\\  
\hline
\bf{F0Q} & Enumeration of defoliated-trees and QAP & \ref{sect_brute} & Used as a baseline\\
\bf{F1Q} & Formulation based on bilinear matrix equation for trees & \ref{sect_LD} & Novel compact formulation\\
\bf{F2Q} &  Bilinear distance-based formulation with binary distance variables &  \ref{sect_quad_binary} & Based on ideas from \cite{veremyev2015critical} \\
\hline
\multicolumn{4}{c}{Mixed-integer linear formulations}\\  
\hline
\bf{F0L} & Flow-based formulation & \ref{sect_flow} & Multicommodity flow model \\
\bf{F1L} & Distance-based formulation with integral distance variables & \ref{sect_f1l} & Novel formulation\\
\bf{F2L} & Distance-based formulation with binary distance variables & \ref{sect_f2l} & Linearization of \bf{F2Q} \\
\hline
\end{tabular}
\end{center}
\label{tb_formulations}
\end{table}%

Some variations of basic formulations were considered: 
\begin{enumerate}
    \item The refined formulation  \textbf{F1Q$^\urcorner$} with the reduced number of equality constraints was compared to the full \textbf{F1Q} formulation. 
    \item For \textbf{F1L} we studied the influence of the big-$M$ constant in \eqref{f1l_dij_2}, trying the values $L$ and $n-1$. When applicable, two versions of \textbf{F1L} formulation are denoted as \textbf{F1L}$(L)$ and \textbf{F1L}$(n-1)$ respectively. 
\end{enumerate}

We used Gurobi\textsuperscript{TM} V.\,9.0.0 invoked from Python 3 Anaconda environment under Windows Server 2008 R2 running on Intel(R) Xeon(R) X5670 (2.93 GHz, 6 cores, 12 threads) with 12 Gb RAM. Standard Gurobi parameters were used, including the absolute MIP tolerance $10^{-4}$. 

\subsection{Initialization of the branch-and-bound algorithm}\label{sect_heur}

As many NP-complete problems, OCSTP is solved either exactly with limited enumeration or approximately using heuristic algorithms. This article focuses on the exact methods, and heuristics are neither surveyed nor studied in detail. At the same time, since initialization of the branch-and-bound algorithm with some heuristic solution (a record) considerably affects its performance, in our computational experiments we compare two initialization strategies. 

The first strategy employs general-purpose MIP heuristics implemented in contemporary combinatorial optimization packages (in our experiments we used Gurobi 9.0). In the second strategy we seek to improve this default initialization using the local search algorithm from \cite{goubko2020lower}. However, in both strategies a part of computation time was invested into record improvement during the execution of the branch-and-bound algorithm according to standard policies of the optimization package. 
In algorithms of the local search every feasible solution is assigned a subset of feasible solutions called the ``neighborhood''.  The algorithm starts from an arbitrary feasible solution and iteratively transits to the least-cost solution from the neighborhood until no cost improvement is obtained.

Many algorithms of this class are developed for OCSTP. Our algorithm differs in two aspects. First, it is developed to design trees with the given degree sequence, so elementary graph transformations that form the neighbourhood of a given candidate tree preserve the degree sequence. Second, these elementary graph transformations are used in the proof of Huffman tree optimality for the Wiener index for vertex-weighted trees in \cite{goubko2016minimizing}. So, this algorithm converges to the optimal spanning tree at least when the requirements matrix is represented in the form $A=\mu\cdot\mu^\top$, where $\mu$ is a non-negative vector of vertex weights\footnote{For the Huffman tree to be optimal, $\mu$ should also be \textit{consistent} with $d$, i.,e., $d_i<d_j$ implies $\mu_i\le \mu_j$, $\forall i,j\in V$.}. Performance and computation time are not discussed here (see \cite{goubko2020lower} for details). 

Basing on our computational experiments, below we report the comparative quality of our local search heuristic with respect to the general-purpose MIP heuristics. We also investigate the effect of initial solution quality on the computation time for different formulations of OCSTP. 

\subsection{Comparing MIQP and MILP formulations on small-size problems}

First we compare all variants of formulations from Table \ref{tb_formulations} on a set of eight small requirement matrices of order $n=15$ obtained from all eight source matrices from Table \ref{tb_datasets}. This (rather small) order is selected so that the branch-and-bound algorithm for all formulations converges under a reasonable time limit of 15\,000~sec. The same collection of ten vertex degree sequences was tested against each requirements matrix giving the total amount of 80 distinct cases. 

\textbf{F1Q$^\urcorner$} appears to be 1.5 times faster than \textbf{F1Q} in average, while  \textbf{F1L}$(L)$ and \textbf{F1L}$(n-1)$ exhibit roughly equal performance. The comparison of the best versions of six formulations is presented on the ``box with whiskers'' plot in Figure~\ref{fig_timeALLmodels15}. The average and median computation time along with algorithms' ranking are also shown in Table \ref{tb_timeALLmodels15}.

\begin{figure}[th]
\centering
\includegraphics[scale=0.8, trim=0pt 20pt 0pt 0pt, clip]{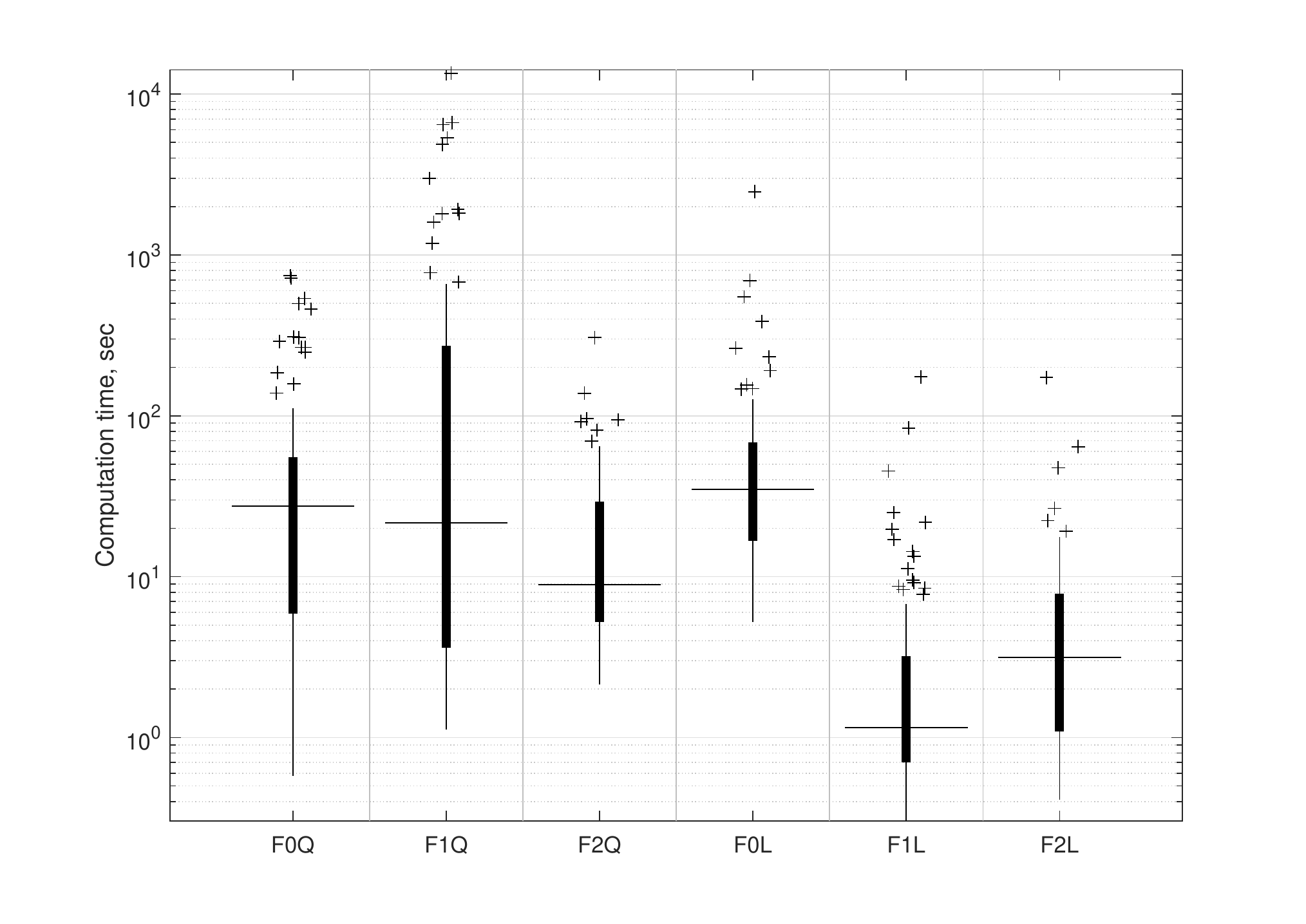} 
\caption{Distribution of computation time for MIQP and MILP formulations, $n=15$. The horizontal mark represents the median computation time, the black bar covers a quartile above and a quartile below the median, and the vertical line covers the whole range of computation time except individual outliers shown as cross marks. The logarithmic scale is used for better visibility.}
\label{fig_timeALLmodels15}
\end{figure}

\begin{table}[ht]
\scriptsize
\setlength{\tabcolsep}{3pt}
\renewcommand{\arraystretch}{1.2}
\caption{Summary of computation time for MIQP and MILP formulations, $n=15$.}
\begin{center}
\begin{tabular}{ l|c|c|c|c|c|c }
\hline
Formulation & F0Q & F1Q & F2Q  & F0L & F1L & F2L \\
\hline
Median computation time, sec. & 24.86 & 21.59 & 8.95 & 34.80 & 1.16 & 3.14 \\
\hline
Median time rank & V & IV & III & VI & I & II \\
\hline
Average computation time, sec. & 79.40 & 683.23 & 25.08 & 99.15 & 7.06 & 8.35 \\
\hline
Average time rank & IV & VI & III & V & I & II \\
\hline
\end{tabular}
\end{center}
\label{tb_timeALLmodels15}
\end{table}%

These experiments show that the novel distance-based MILP formulations \textbf{F1L} and \textbf{F2L} are far ahead of both the classical multicommodity flow formulation \textbf{F0L} and of all three MIQP formulations, at least, for $n=15$. The median example is solved with \textbf{F1L} 3 times faster than with \textbf{F2L}. However, the computation time of \textbf{F1L} is less stable, and, due to several long-running outlier examples, \textbf{F1L} and \textbf{F2L} enjoy almost equal average computation time. 

\textbf{F2Q} is just a bilinear version of \textbf{F2L}. It is several times slower than \textbf{F2L} due to less economical default linearization performed by Gurobi software. However, \textbf{F2Q} appears more efficient than two other MIQP formulations and even than \textbf{F0L}. 

Complex equality constraints in \textbf{F1Q} result in longer computation and in the less stable computation time (the solution improvement process sometimes stagnates). While \textbf{F1Q$^\urcorner$} is faster than \textbf{F1Q} for $n=15$, its superiority should be verified on larger examples. 
 
Unexpectedly, the performance of the MILP multicommodity-flow formulation \textbf{F0L} is worse than that of the MIQP tree-enumeration algorithm \textbf{F0Q}. This conclusion, however, should also be verified on on larger examples.

For small dimensions the local search heuristics has reasonably good quality, with the average gap less than 4\%. However, it has a surprisingly small effect on the computation time of algorithms. As shown in Figure~\ref{fig_heur}, the median computation time even increases a bit for \textbf{F1Q}, \textbf{F2Q} and \textbf{F1L}. For \textbf{F0L} and \textbf{F2L} the gain from using a local search heuristics is also negligible. However, for the larger problem dimensions the heuristics might be still useful (see Section~\ref{sect_comp_large}).

\begin{figure}[th]
\centering
\includegraphics[scale=0.8, trim=0pt 45pt 0pt 0pt, clip]{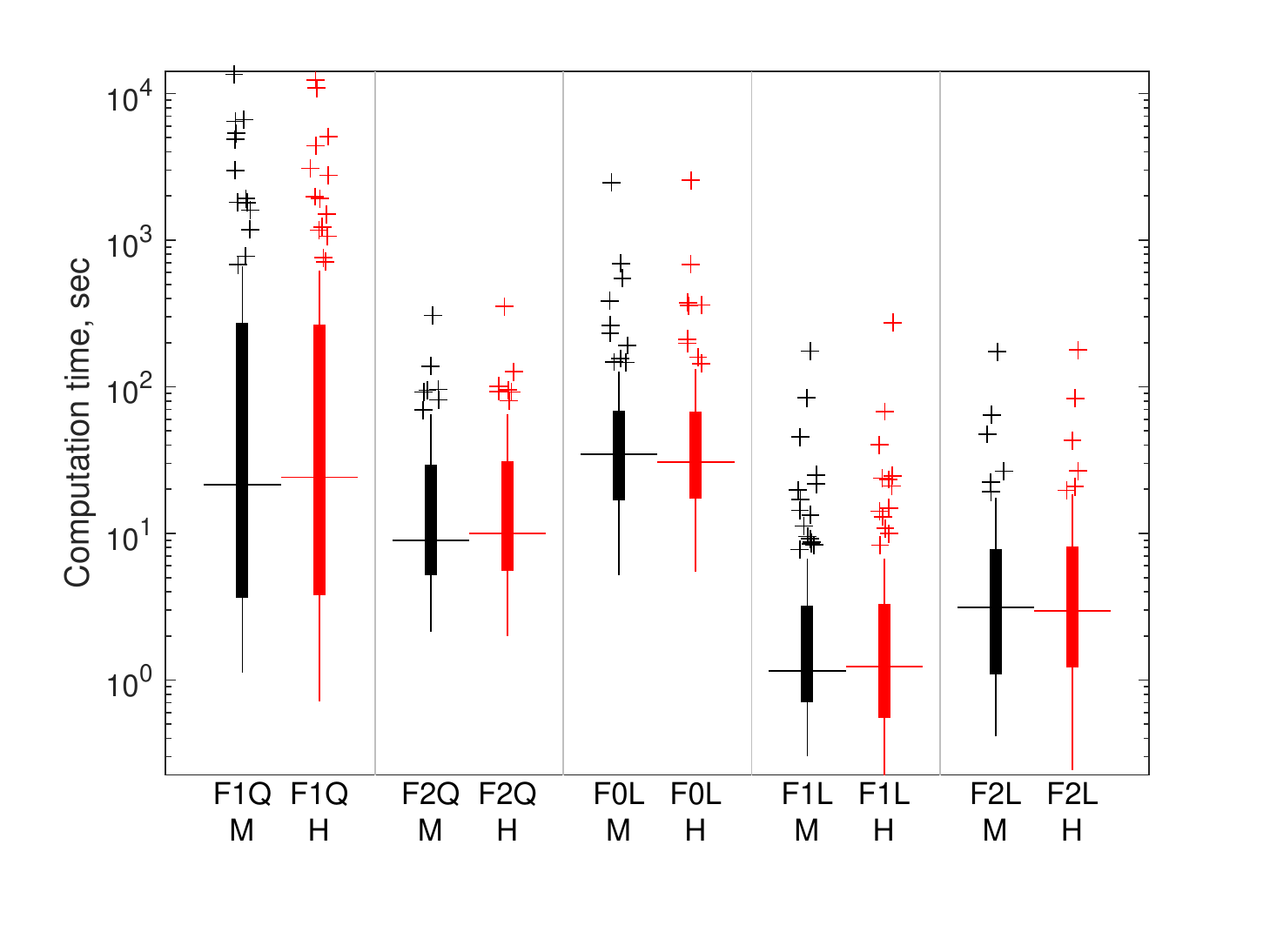} 
\caption{Comparing the computation time distribution with MIP heuristics (\textbf{black}) and with the local search heuristics (see Section~\ref{sect_heur}) provided ({\color{red}\textbf{red}}), $n=15$. The legend is the same as in Figure~\ref{fig_timeALLmodels15}.}
\label{fig_heur}
\end{figure}

For all formulations the computation time has high dispersion, not only because of the unpredictability of the branch-and-bound schemes, but merely because the volume of the search space crucially depends on the vertex degree sequence. A typical effect of the vertex degree sequence is shown in Figure~\ref{fig_degree}, where \textbf{F0Q} and \textbf{F0L} formulations are compared. The computation time is highly correlated for all algorithms (e.g., for the leading formulations \textbf{F1L} and \textbf{F2L} the correlation is 0.99), so the vertex degree sequence is the main factor of the problem complexity. 

\begin{figure}[th]
\centering
\includegraphics[scale=0.8]{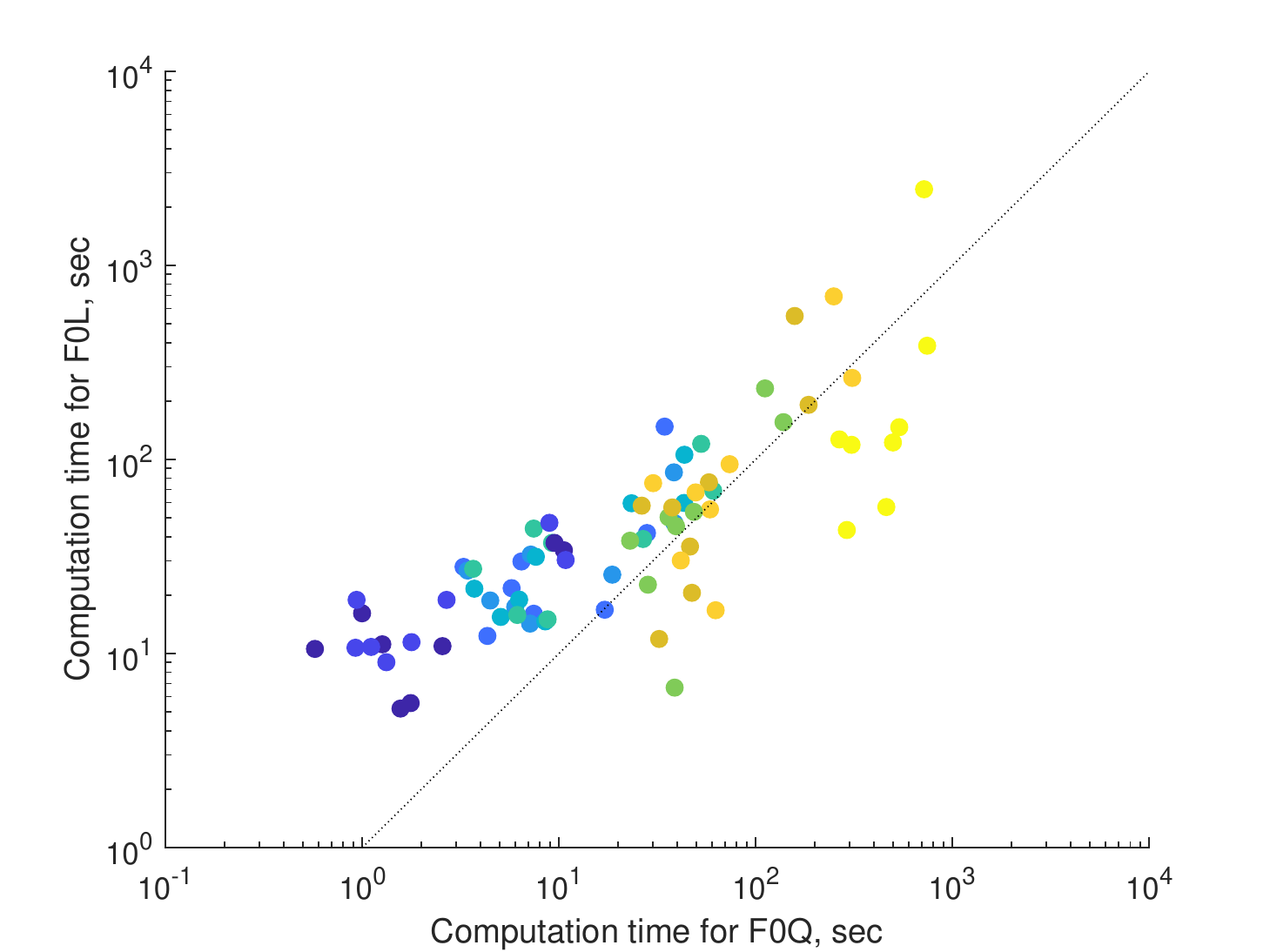}
\caption{Computation time: \textbf{F0L} vs \textbf{F0Q}, $n=15$ (double logarithmic scale). Different colors are used for vertex degree sequences from the simplest (violet) to the most complex (bright yellow). \textbf{F0L} is faster for cases depicted below the diagonal (the dotted line), while \textbf{F0Q} is faster in all other cases.}
\label{fig_degree}
\end{figure}

\subsection{Comparing MILP formulations on medium-size problems}\label{sect_comp_medium}

The next series of experiments were run for requirements matrices of bigger order $n=20$.  The ``BCC'' data set cannot provide different submatrices of size $20$, so it is discarded. With the increase of the problem dimension, formulation \textbf{F1Q} becomes too slow, so we exclude it from consideration. We also exclude \textbf{F2Q}, since it is inferior to its careful linearization \textbf{F2L}. Again, the time limit is 15\,000~sec. Ten vertex degree sequences were tested against each requirements matrix giving the total amount of 70 distinct cases. 

The results are presented in Figure \ref{fig_time15000}. It shows that the gap between the novel distance-based formulations (\textbf{F1L} and \textbf{F2L}) and the multicommodity-flow formulation \textbf{F0L} does not decrease with the problem dimension. Both for $n=15$ and $n=20$ \textbf{F0L} is approximately 35 times slower than \textbf{F1L}. \textbf{F1L} is still faster than \textbf{F2L}. However, the difference between their median computation time is smaller than that for the smaller problem dimension $n=15$. In the next subsection we compare the two distance-based MILP formulations in more detail.

\begin{figure}[th]
\centering
\includegraphics[scale=0.8, trim=0pt 10pt 0pt 0pt, clip]{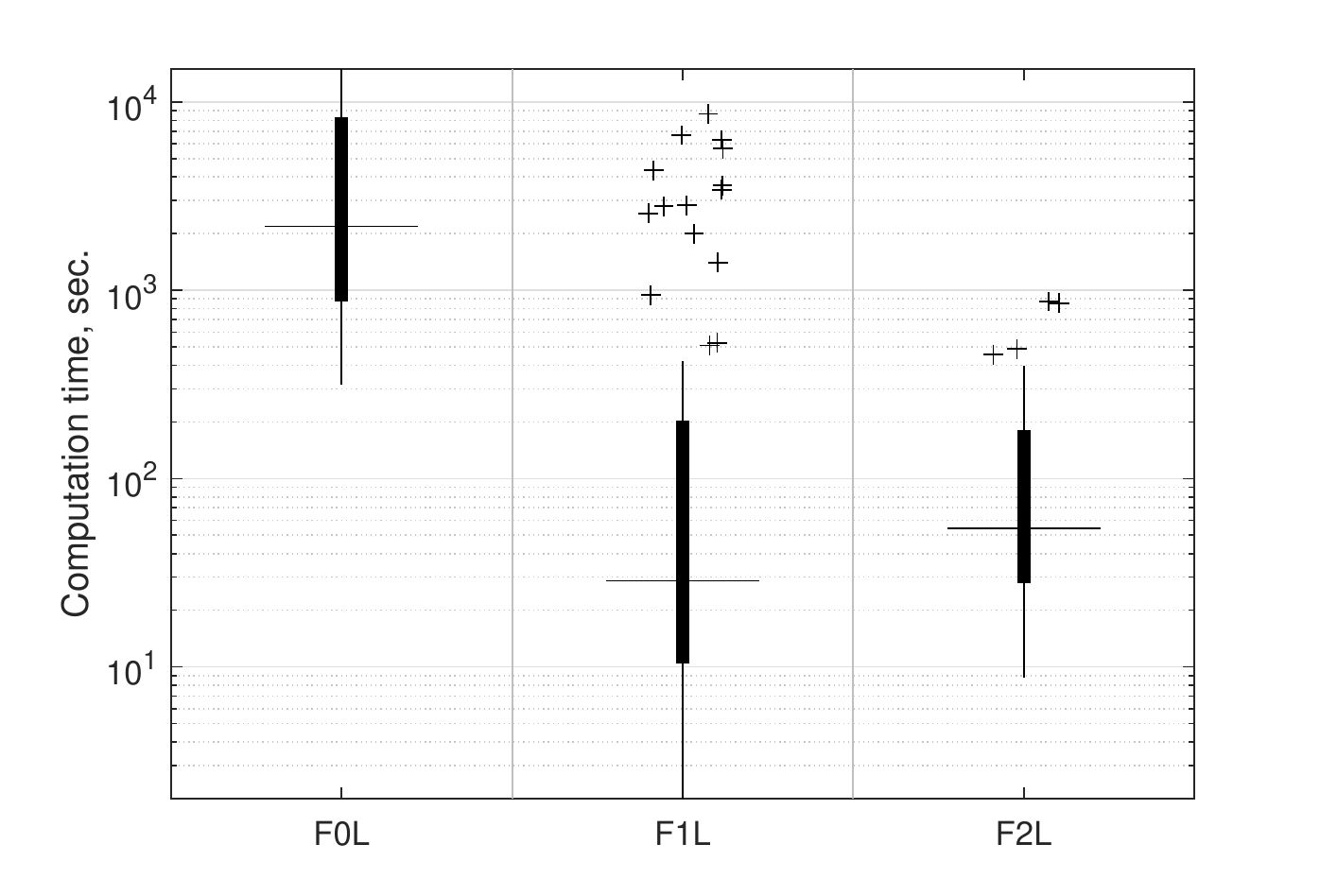} 
\caption{Distribution of computation time for MILP formulations, $n=20$. The legend is the same as in Figure~\ref{fig_timeALLmodels15}. The logarithmic scale is used for better visibility.}
\label{fig_time15000}
\end{figure}

\subsection{Comparing distance-based MILP formulations on large-size problems}\label{sect_comp_large}

The two leading formulations, \textbf{F1L} and \textbf{F2L}, are efficient enough to solve in reasonable time most instances of the constrained OCSTP with dimension up to 30. To identify the relation between the problem dimension and the computation time, \textbf{F1L}$(L)$ and \textbf{F2L} where run for all $n=15,...,28$ and all available data sets from Table~\ref{tb_datasets} with the starting point provided by the standard MIP heuristics. As before, ten distinct vertex degree sequences were generated for each problem dimension and tested against each requirements table of the appropriate size.

Again, the computation time has considerable dispersion for each problem dimension, so we present the quartile distributions in Figure~\ref{fig_time1000}. The results are a bit confusing: while \textbf{F1L} is faster for $n=15,...,20$, for $n \ge 21$ \textbf{F2L} takes lead. As more than a half of \textbf{F1L} examples where timed out for $n\ge 23$, we performed a limited testing for $n=15,20,25,30$ with the time limit increased to $80\,000$ sec. The results presented in Figure~\ref{fig_time80000} show the comparative performance of distance-based MILP formulations more clearly. Only 19\% of examples for $n=30$ where not timed out for \textbf{F1L}, while 63\% of examples converged to optimality for \textbf{F2L}. For examples of dimension $n=30$ where both \textbf{F1L} and  \textbf{F2L} converged to optimality, \textbf{F2L} is 8 times faster in average. Figure~\ref{fig_time80000} also shows that the computation time for \textbf{F2L} is more stable. 

\begin{figure}[ht!]
\centering
\includegraphics[scale=0.8, trim=0pt 40pt 0pt 0pt, clip]{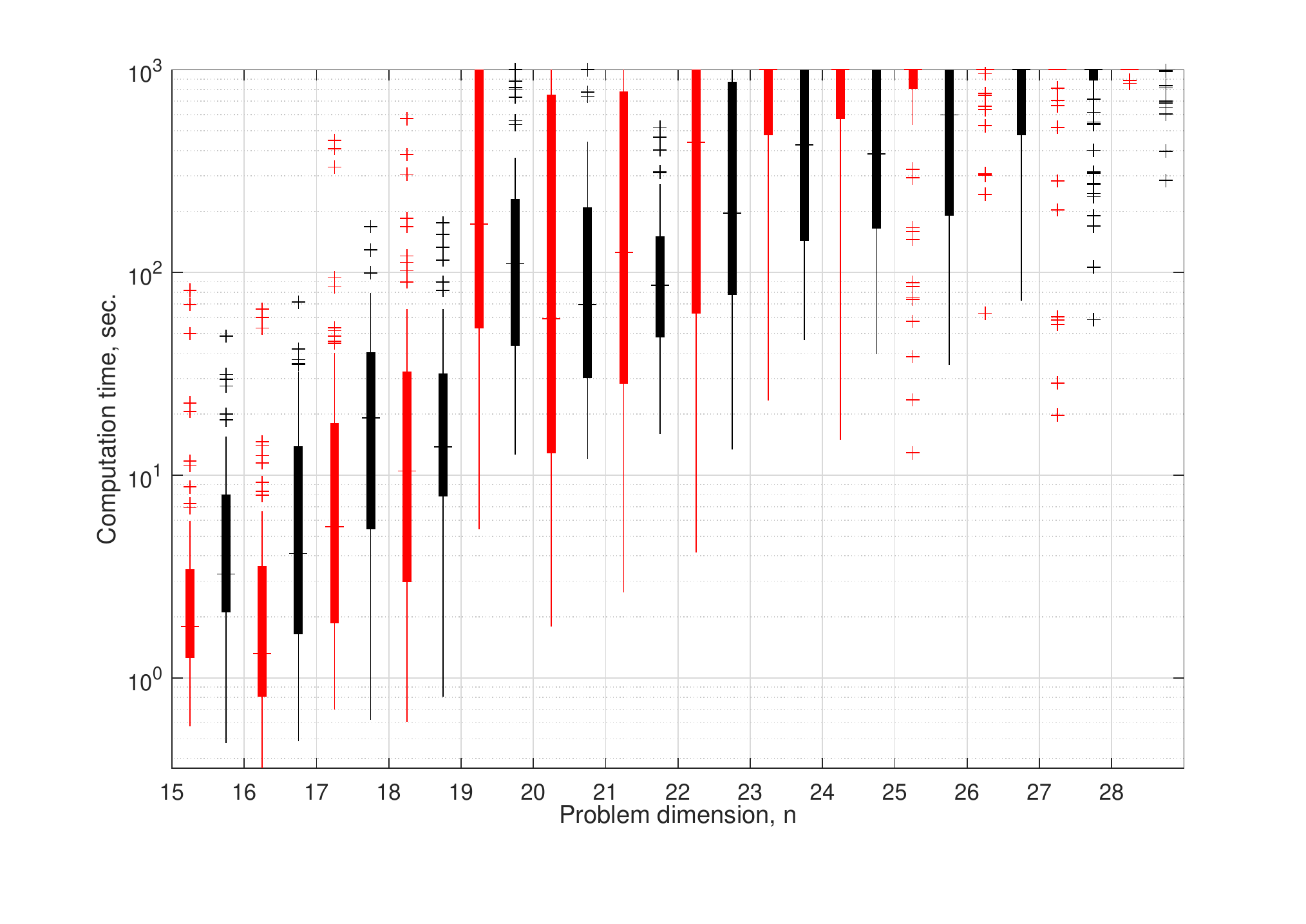} 
\caption{Distribution of computation time for \textbf{F1L} ({\color{red}\textbf{red}}) and \textbf{F2L} (\textbf{black}), $n=15...30$. The legend is the same as in Figure~\ref{fig_timeALLmodels15}. The logarithmic scale is used for better visibility.}
\label{fig_time1000}
\end{figure}

\begin{figure}[th]
\centering
\includegraphics[scale=0.8, trim=0pt 30pt 0pt 0pt, clip]{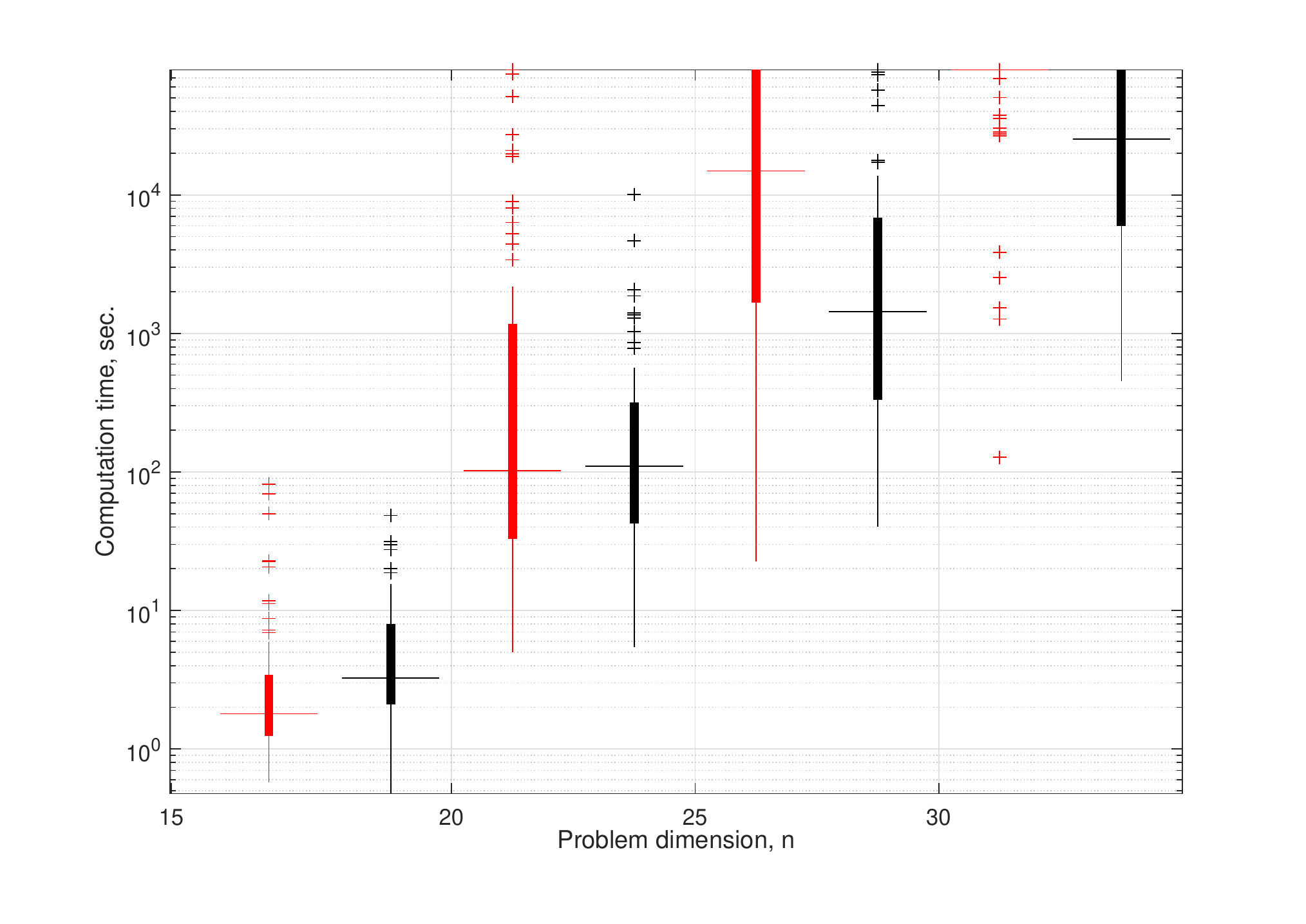} 
\caption{Distribution of computation time for \textbf{F1L} ({\color{red}\textbf{red}}) and \textbf{F2L} (\textbf{black}), $n=15, 20, 25, 30$ time limit increased to 80\,000 sec. The legend is the same as in Figure~\ref{fig_timeALLmodels15}. The logarithmic scale is used for better visibility.}
\label{fig_time80000}
\end{figure}

Therefore, although \textbf{F2L} uses $L$ times more binary variables than \textbf{F1L}, the former outperforms the latter for large-sized examples because it provides better relaxations (We compare the best lower bounds obtained under both formulations in Figure~\ref{fig_lb}.) 

\begin{figure}[ht!]
\centering
\includegraphics[scale=0.8, trim=0pt 10pt 0pt 0pt, clip]{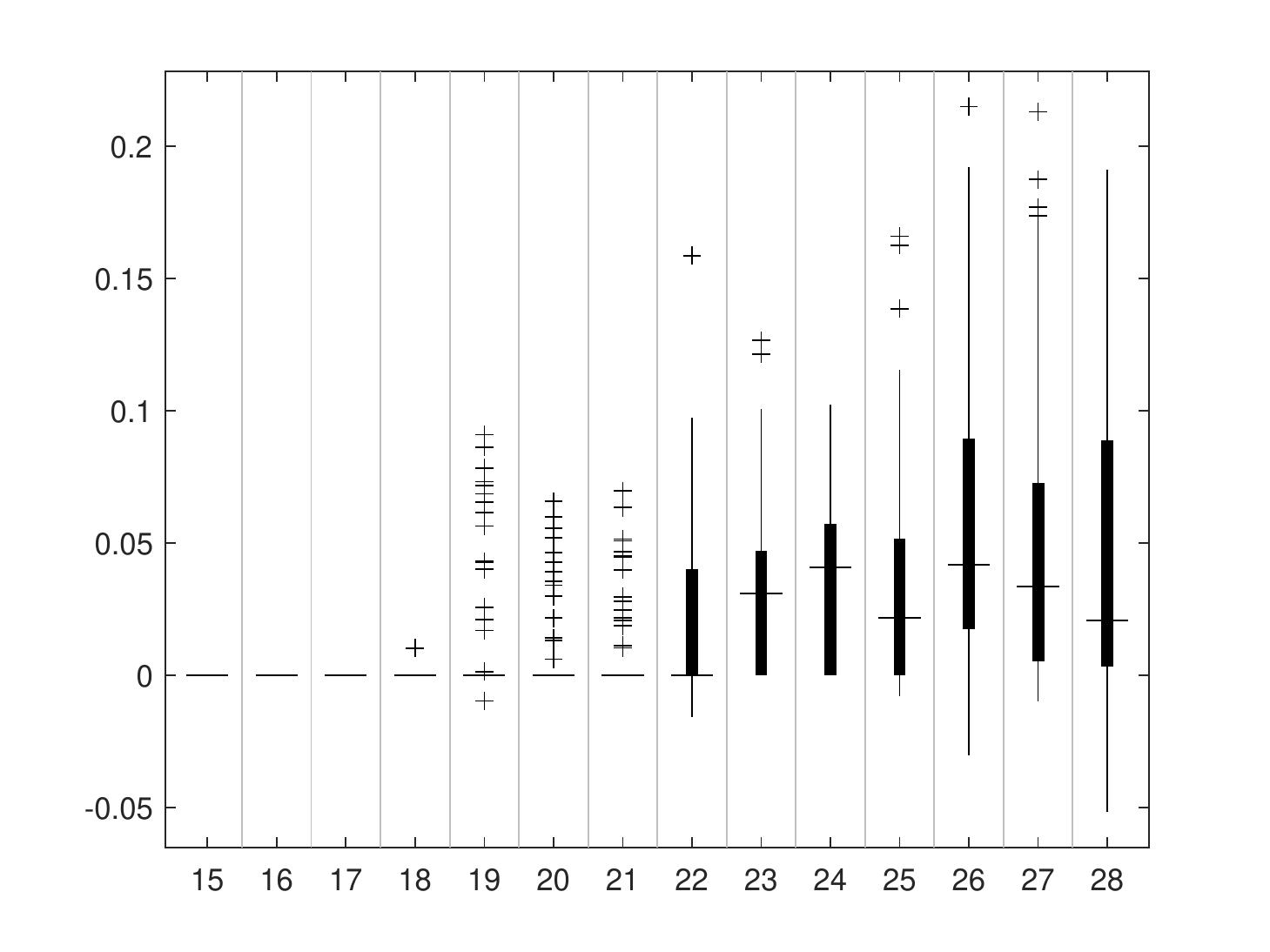} 
\caption{The ratio $\frac{LB(\mathbf{F2L})}{LB(\mathbf{F1L})}-1$ of the best lower bounds ($LB$) obtained under \textbf{F1L} and \textbf{F2L} formulations until the time limit reached. The results show that \textbf{F2L} almost always provides the better bound.}
\label{fig_lb}
\end{figure}

At the same time, \textbf{F1L} provides better records (the best found solutions) than \textbf{F2L}. From Figure~\ref{fig_rec} we see that the best solution found by \textbf{F1L} is often much better than the one found by \textbf{F2L}. To be precise, \textbf{F1L} finds better solution two times more often than \textbf{F2L}. Moreover, in 9\% of examples 1000 sec was not enough for \textbf{F2L} to outperform even the local search heuristic solution. The effect strengthens with the increase of problem dimension (for $n=28$ \textbf{F2L} outperforms the heuristic solution only in 38\% of examples). We can conclude that \textbf{F1L} appears more convenient for MIP heuristic algorithms implemented in Gurobi software. 

Therefore, if one is interested in the best solution found in limited time, the formulation \textbf{F1L} clearly wins the competition. Moreover, the effect of the better record prevails the effect of the better lower bound, and \textbf{F1L} also has the lower average \textit{gap}\footnote{The ratio $\frac{\mathrm{Record}}{LB}-1$ representing the relative length of the range that encloses the optimal solution.}, 3.8\% against 4.3\% for \textbf{F2L}. 

\begin{figure}[th]
\centering
\includegraphics[scale=0.8, trim=0pt 10pt 0pt 0pt, clip]{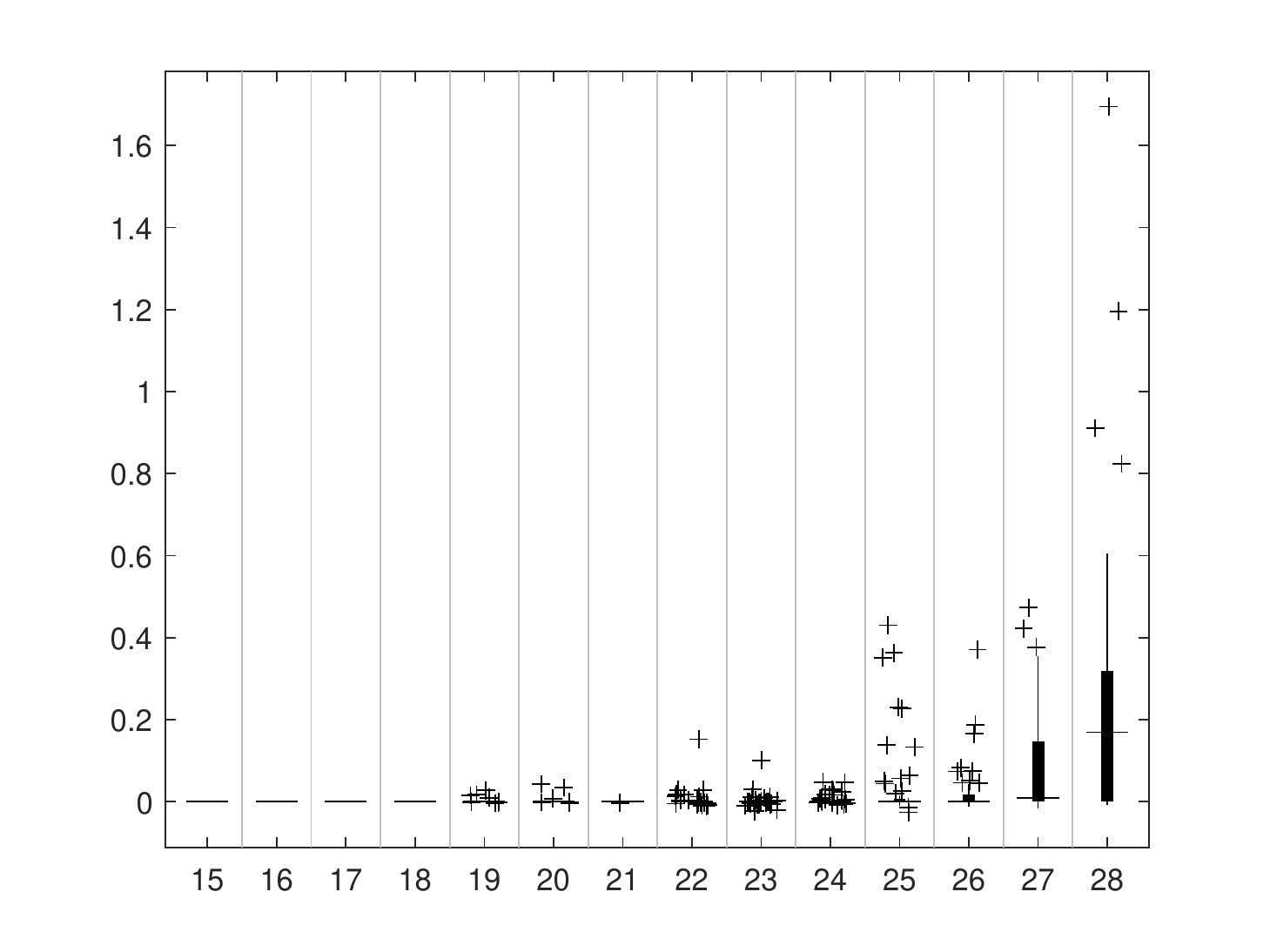} 
\caption{The ratio $\frac{\textrm{Record}(\mathbf{F2L})}{\textrm{Record}(\mathbf{F1L})}-1$ of the best solution costs (Record) found under \textbf{F1L} and \textbf{F2L} formulations until the time limit of 1000 sec is reached. The results show that \textbf{F1L} almost always finds better feasible solutions.}
\label{fig_rec}
\end{figure}

Finally, let us note that the computation time of the two formulations is highly correlated (H=0.72 for all cases converged to optimality), so, examples being difficult to \textbf{F1L} also appear difficult to \textbf{F2L}.

\section{Conclusion}

The main aim of this article was to break the monopoly of flow-based approaches to OCSTP by introducing several MIQP and MILP formulations based on the direct modeling of a distance matrix of a tree by continuous or binary variables. 

We proposed several simple, intuitive, flexible and easy to implement distance-based MILP and MIQP models for OCSTP and compared their performance on a wide variety of real-world samples from transportation industry.

It is important to note that these novel formulations, namely, \textbf{F1Q}, \textbf{F1L}, and \textbf{F2Q} (considered together with its linearization \textbf{F2L}) can be generalized beyond the scope of the constrained ORSTP, which was used as a touchstone in this article: 
\begin{itemize}
\item All formulations except \textbf{F1Q} can be extended to the general network design problems. The distance between vertices in a tree then becomes the shortest-path distance in a graph with loops. Connected graphs with loops can be sought by just allowing for non-arborescent vertex degree sequences as an input for \textbf{F1L}, \textbf{F2Q} (and \textbf{F2L}).
\item All formulations directly extend to directed graphs by weakening the symmetry conditions on structural variables (e.g., \eqref{f1q_symmx} and alike). Certainly, the vertex degree constraints should be revised, if relevant.
\item \textbf{F1Q} and \textbf{F1L} are extended to the general OCSTP with variable graph edge lengths $t_{ij}$, $(i,j)\in E_s$ as explained in Sections \ref{sect_LD} and \ref{sect_f1l}. \textbf{F2L} can also be extended to integral lengths $t_{ij}$ in a quasi-polynomial manner as proposed in \cite{veremyev2015critical}.
\item According to \cite{goubko2020bilinear}, the system of equations \eqref{eq_LD} completely characterizes the set of (weighted) trees. Therefore, the optimization criterion in \textbf{F1Q} need not be a monotone function of graph distances. The formulation \textbf{F2L} has also been initially proposed for a non-monotone criterion. However, additional constraints should be added to \textbf{F2L} (and, accordingly, to \textbf{F2Q}) as explained in \cite{veremyev2015critical}. 
\item A common drawback of the novel distance-based formulations is that they (in contrast to the multicommodity flow model \textbf{F0L}) hardly integrate the \textit{flow capacity constraints} \citep{magnanti1984network} and edge design or travel costs that depend on the amount of requirement traversing the edge. These features are typical for road topology network design problems \citep{jia2019review}.
\end{itemize}

Table \ref{tb_extend} summarizes the generalization capabilities of the considered formulations for OCSTP.   

\begin{table}[ht]
\setlength{\tabcolsep}{3pt}
\scriptsize
\renewcommand{\arraystretch}{1.2}
\caption{Possible extensions of MIQP and MILP formulations studied in this article}
\begin{center}
\begin{tabular}{ l|c|c|c|c }
\hline
Feature & F0L & F1Q & F1L & F2Q/F2L \\
\hline
Graphs with loops & $+$ & $-$ & $+$ & $+$ \\
\hline
Directed graphs & $+$ & $-$ & $+$ & $+$ \\
\hline
Variable edge lengths & $+$ & $+$ & $+$ & $+/-$ \\
\hline
Non-monotone criterion & $-$ & $+$ & $-$ & $+$ \\
\hline
Flow capacity constraints & $+$ & $-$ & $-$ & $-$ \\
\hline
\end{tabular}
\end{center}
\label{tb_extend}
\end{table}%

%


%

Finally, we list several minor improvements that might refine the formulations.
\begin{itemize}
    \item We can tighten the constraints \eqref{f1q_dij_var}, \eqref{f1l_dij_lim} by replacing the maximum tree diameter $L$ by an adaptive upper-bound estimate $L_{ij}$ for graph distances. Namely, if both $i$ and $j$ are internal vertices, then $L_{ij}=L-2$. If one of them is a leaf and the other is an internal vertex, then $L_{ij}=L-1$, and $L_{ij}=L$ otherwise.

    \item Compared to standard out-of-the-box procedures for candidate solution generation, custom local search heuristics from \cite{goubko2020lower} generally (although not in all cases) give slight improvement in computation time and convergence speed. Therefore, it might be a promising idea to combine different heuristics calculated in parallel at the pre-processing stage to speed up all considered formulations.

    \item There is some space for the improvement of the tree-enumeration algorithm from  Section~\ref{sect_brute}. Although defining numerous QAP problems for Gurobi never takes more than 10\% of total computation time (even for small problems with $n<20$), the multiple-scenario feature of Gurobi software will be able to foster calculation of a collection of QAP problems that share the same structure. Specialized QAP solvers might further improve the leaf assignment step. Also, QAP problems of leaf assignment for different tree topologies could be solved in parallel to decrease the overall computation time. 
\end{itemize}

\section*{Acknowledgements}

This work was supported by the Russian Foundation for Basic Research (RFBR) [18-07-01240].

The authors are grateful to the Gurobi Optimization LLC for the academic licenses of their optimization software, which was extensively used in computational experiments. 

\bibliographystyle{apalike}
\bibliography{references}

\end{document}